\documentclass[11pt]{article}

\usepackage{epsfig,epsf,fancybox}
\usepackage{amsmath}
\usepackage{mathrsfs}
\usepackage{amssymb}
\usepackage{graphicx}
\usepackage{color}
\usepackage[linktocpage,pagebackref,colorlinks,linkcolor=blue,anchorcolor=blue,citecolor=blue,urlcolor=blue]{hyperref}

\newtheorem{theorem}{Theorem}[section]

\newtheorem{lemma}[theorem]{Lemma}
\newtheorem{proposition}[theorem]{Proposition}
\newtheorem{definition}{Definition}[section]
\newtheorem{example}{Example}[section]

\newcommand{\Br}{\mathbb{R}}
\newcommand{\T}{\top}

\newcommand{\st}{\textnormal{s.t.}}
\newcommand{\tr}{\textnormal{Tr}\,}

\newcommand{\Diag}{\textnormal{Diag}\,}

\newcommand{\rank}{\textnormal{rank}}

\newcommand{\matr}{\boldsymbol M}\,
\newcommand{\vect}{\boldsymbol V}\,

\newcommand{\SI}{\mathbf S}
\newcommand{\RR}{\mathbf R}
\newcommand{\argmin}{\mathop{\rm argmin}}
\newcommand{\LCal}{\mathcal{L}}
\newcommand{\CCal}{\mathcal{C}}
\newcommand{\DCal}{\mathcal{D}}
\newcommand{\half}{\frac{1}{2}}
\newcommand{\K}{\mbox{$\mathbb K$}}
\newcommand{\Z}{\mbox{$\mathbb Z$}}

\newcommand{\br}{\mathbb{R}}

\newcommand{\be}{\begin{equation}}
\newcommand{\ee}{\end{equation}}
\newcommand{\ba}{\begin{array}}
\newcommand{\ea}{\end{array}}
\newcommand{\bpm}{\begin{pmatrix}}
\newcommand{\epm}{\end{pmatrix}}
\newcommand{\ACal}{\mathcal{A}}
\newcommand{\FCal}{\mathcal{F}}

\newcommand{\etal}{{et al. }}

\begin{document}

\title{Tensor Principal Component Analysis via Convex Optimization}

\author{Bo JIANG
\thanks{Department of Industrial and Systems Engineering, University of Minnesota, Minneapolis, MN 55455.{ Research of this author was supported in part by the National Science Foundation under Grant Number CMMI-1161242.}}
\and Shiqian MA
\thanks{Department of Systems Engineering and Engineering Management, The Chinese University of Hong Kong, Shatin, N. T., Hong Kong.{ Research of this author was supported in part by a Direct Grant of the Chinese University of Hong
Kong (Project ID: 4055016) and the Hong Kong Research Grants Council
(RGC) Early Career Scheme (ECS) (Project ID: CUHK 439513).}}
\and Shuzhong ZHANG
\thanks{Department of Industrial and Systems Engineering, University of Minnesota, Minneapolis, MN 55455.{ Research of this author was supported in part by the National Science Foundation under Grant Number CMMI-1161242.}}}

\date{\today}

\maketitle

\begin{abstract}

This paper is concerned with the computation of the principal components for a general tensor, known as the tensor principal component analysis (PCA) problem. We show that the general tensor PCA  problem is reducible to its special case where the tensor in question is super-symmetric with an even degree. In that case, the tensor can be embedded into a symmetric matrix. We prove that if the tensor is rank-one, then the embedded matrix must be rank-one too, and vice versa. The tensor PCA problem can thus be solved by means of matrix optimization under a rank-one constraint, for which we propose two solution methods: (1) imposing a nuclear norm penalty in the objective to enforce a low-rank solution; (2) relaxing the rank-one constraint by Semidefinite Programming. Interestingly, our experiments show that both methods can yield a rank-one solution {for almost all the randomly generated instances}, in which case solving the original tensor PCA problem to optimality. To further cope with the size of the resulting convex optimization models, we propose to use the alternating direction method of multipliers, which reduces significantly the computational efforts. Various extensions of the model are considered as well.

\vspace{0.8cm}


\vspace{0.5cm}


\end{abstract}

\newpage

\section{Introduction}\label{introduction}
Principal component analysis (PCA) plays an important role in applications arising from data analysis,
dimension reduction and bioinformatics etc. PCA finds a few linear combinations of the original variables. These linear combinations, which are
called principal components (PCs), are orthogonal to each other and explain most of the variance of the data. PCs provide a powerful tool to compress data along the direction of maximum variance to reach the minimum information loss. Specifically, let $\xi = (\xi_1,\ldots,\xi_m)$ be an $m$-dimensional random vector. Then for a given data matrix
$A\in\br^{m\times n}$ which consists of $n$ samples of the $m$ variables, finding the PC that explains the largest variance of the variables $(\xi_1,\ldots,\xi_m)$ corresponds to the following optimization problem:
\be\label{prob:PCA} (\lambda^*,x^*,y^*) := \min_{\lambda\in\RR, x\in\RR^m, y\in\RR^n}\|A - \lambda x y^{\T}\|. \ee
Problem \eqref{prob:PCA} is well known to be reducible to computing the largest singular value (and corresponding singular vectors) of $A$, and can be
equivalently formulated as:
\be\label{prob:PCA-eigen}
\begin{array}{ll}
\max_{x,y} & \left( \begin{array}{c} x \\ y \end{array} \right)^{\T} \left( \begin{array}{cc} 0 & A  \\ A^{\T} & 0 \end{array} \right) \left( \begin{array}{c} x \\ y \end{array} \right) \\
\mbox{s.t.} & \left\| \left( \begin{array}{c} x \\ y \end{array} \right) \right\| = 1 .
\end{array}
\ee
Note that the optimal value and the optimal solution of Problem \eqref{prob:PCA-eigen} correspond to the largest eigenvalue and the corresponding eigenvector of the symmetric matrix $\bpm 0 & A \\ A^{\T} & 0\epm$.

Although the PCA and eigenvalue problem for matrix have been well studied in the literature, the research of PCA for tensors is still lacking. Nevertheless, the tensor PCA is of great importance in practice and has many applications in computer vision~\cite{WA04},
diffusion Magnetic Resonance Imaging (MRI)~\cite{GTDPMR08,BV08,QYW10}, quantum entanglement problem~\cite{HS10}, spectral hypergraph theory~\cite{HQ12} and higher-order Markov chains~\cite{LN11}. This is mainly because in real life we often encounter
multidimensional data, such as images, video, range data
and medical data such as CT and MRI. A color image can be considered as 3D data with row, column, color in each direction, while a color video sequence can be
considered as $4$D data, where time is the fourth dimension. Moreover, it turns out that it is more reasonable to treat the multidimensional data as a tensor instead of unfolding it into a matrix. For example, Wang and Ahuja~\cite{WA04} reported that the images obtained by tensor PCA technique have higher quality than that by matrix PCA. Similar to its matrix counterpart, the problem of finding the PC that explains the most variance of a tensor $\mathcal{A}$ (with degree $m$) can be formulated as:
\be\label{prob:tensorPCA-eigen}
\begin{array}{ll}
\min & \| \mathcal{A} - \lambda x^1\otimes x^2 \otimes \cdots \otimes x^m \| \\
\mbox{s.t.} & \lambda\in \RR,\, \| x^i\| =1, i=1,2,\ldots,m,
\end{array}
\ee
which is equivalent to
\be\label{prob:tensorPCA-tensor}
\begin{array}{ll}
\max & \mathcal{A} (x^1, x^2 , \cdots , x^m)   \\
\mbox{s.t.} & \| x^i\| =1, i=1,2,\ldots,m,
\end{array}
\ee
where $\otimes$ denotes the outer product between vectors; viz.
$$
(x^1\otimes x^2 \otimes \cdots \otimes x^m)_{i_1 i_2\cdots i_m} = \prod\limits_{k=1}^m (x^k)_{i_k}.
$$

Let us call the above solution the {\em leading}\/ PC.
Once the leading PC is found, the other PCs can be computed sequentially via the so-called ``deflation'' technique. For instance, the second PC is defined as the leading PC of the tensor subtracting the leading PC from the original tensor, and so forth.
The theoretical basis of such a deflation procedure for tensors is not exact sound, 
although its matrix counterpart is well established ({see \cite{Mackey-NIPS-2008} and the references therein for more details}).
However, the deflation process does provide a heuristic way to compute multiple principal components of a tensor, albeit approximately.
Thus in the rest of this paper, we focus on finding the leading PC of a tensor.

Problem \eqref{prob:tensorPCA-tensor} is also known as the best rank-one approximation of tensor $\mathcal A$.
As we shall see later, problem~\eqref{prob:tensorPCA-tensor} can be reformulated as
\begin{equation}\label{prob:tensorPCA-Z-eigen-super-sym}
\begin{array}{ll}
\max & \mathcal{F} (x, x , \cdots , x)   \\
\mbox{s.t.} & \| x \| = 1,
\end{array}
\end{equation}
where $\FCal$ is a super-symmetric tensor.
Problem~\eqref{prob:tensorPCA-Z-eigen-super-sym} is NP-hard and is known as the maximum Z-eigenvalue problem. Note that a variety of eigenvalues and
eigenvectors of a real symmetric tensor were introduced by Lim~\cite{Lim05} and Qi~\cite{Q05} independently in 2005. Since then, various methods have been proposed to find the Z-eigenvalues~\cite{CHLZ11, QWW09, KR02, KB09, KM11}, which however may correspond only to local optimums.
In this paper, we shall focus on finding the global optimal solution of \eqref{prob:tensorPCA-Z-eigen-super-sym}.

Before proceeding
let us introduce notations that will be used throughout the paper. We denote $\RR^n$ to be the $n$-dimensional Euclidean space. A tensor is a high dimensional array of real data, usually in calligraphic letter, and is denoted as $\mathcal{A}=(\mathcal{A}_{i_1i_2\cdots i_m})_{n_1\times n_2\times \cdots\times n_m}$. The space where $n_1\times n_2\times \cdots\times n_m$ dimensional real-valued tensor resides is denoted by $\RR^{n_1\times n_2\times \cdots\times n_m}$. We call $\mathcal{A}$ super-symmetric if $n_1=n_2=\cdots=n_m$ and $\mathcal{A}_{i_1i_2\cdots i_m}$ is invariant under any permutation of $(i_1,i_2,...,i_m)$, i.e., $\ACal_{i_1i_2\cdots i_m} = \ACal_{\pi(i_1,i_2,\cdots,i_m)}$, where $\pi(i_1,i_2,\cdots,i_m)$ is any permutation of indices $(i_1,i_2,\cdots,i_m)$. The space where $\underbrace{n\times n \times \cdots \times n}_{m}$ super-symmetric tensors reside is denoted by $\SI^{n^m}$. Special cases of tensors are vector ($m=1$) and matrix ($m=2$), and tensors can also be seen as a long vector or a specially arranged matrix. For instance, the tensor space $\RR^{n_1\times n_2\times \cdots\times n_m}$ can also be seen as a matrix space
$\RR^{(n_1\times n_2\times \cdots\times n_{m_1}) \times (n_{m_1+1}\times n_{m_1+2} \times \cdots\times n_{m})}$, where the row is actually an $m_1$ array tensor space and the column is another $m-m_1$ array tensor space. Such connections between tensor and matrix re-arrangements will play an important role in this paper. As a convention in this paper, if there is no other specification we shall adhere to the Euclidean norm (i.e.\ the $L_2$-norm) for vectors and tensors; in the latter case, the Euclidean norm is also known as the Frobenius norm, and is sometimes denoted as $\|\mathcal{A}\|_F=\sqrt{\sum_{i_1,i_2,...,i_m} \mathcal{A}_{i_1i_2\cdots i_m}^2}$. For a given matrix $X$, we use $\|X\|_*$ to denote the nuclear norm of $X$, which is the sum of all the singular values of $X$.
Regarding the products, we use $\otimes$ to denote the outer product for tensors; that is, for $\mathcal{A}_1\in \RR^{n_1\times n_2\times \cdots\times n_{m}}$ and $\mathcal{A}_2\in \RR^{n_{m+1}\times n_{m+2} \times \cdots\times n_{m+\ell}}$, $\mathcal{A}_1 \otimes \mathcal{A}_2$ is in $\RR^{n_1\times n_2\times \cdots\times n_{m+\ell}}$ with
\[
(\mathcal{A}_1 \otimes \mathcal{A}_2)_{i_1i_2\cdots i_{m+\ell}} = (\mathcal{A}_1)_{i_1i_2\cdots i_m} (\mathcal{A}_2)_{i_{m+1}\cdots i_{m+\ell}}.
\]
The inner product between tensors $\ACal_1$ and $\ACal_2$ residing in the same space $\RR^{n_1\times n_2\times \cdots \times n_m}$ is denoted
\[
\mathcal{A}_1 \bullet \mathcal{A}_2 = \sum_{i_1,i_2,...,i_m} (\mathcal{A}_1)_{i_1i_2\cdots i_m} (\mathcal{A}_2)_{i_1i_2\cdots i_m}.
\]
Under this light, a multi-linear form $\mathcal{A}(x^1,x^2,...,x^m)$ can also be written in inner/outer products of tensors as
\[\mathcal{A}\bullet (x^1\otimes \cdots \otimes x^m) := \sum\limits_{i_1,\cdots, i_m }\mathcal{A}_{i_1,\cdots, i_m }(x^1\otimes \cdots \otimes x^m)_{i_1,\cdots, i_m } = \sum\limits_{i_1,\cdots, i_m }\mathcal{A}_{i_1,\cdots, i_m }\prod\limits_{k=1}^{m}{x^k_{i_k}}.\]

In the subsequent analysis, for convenience we assume $m$ to be even; i.e., $m=2d$ in \eqref{prob:tensorPCA-Z-eigen-super-sym}, where $d$ is a positive integer. As we will see later, this assumption is essentially non-restrictive. Therefore, we will focus on the following problem of computing the largest eigenvalue of an even order super-symmetric tensor:
\begin{equation}\label{prob:tensor-pca}
\begin{array}{ll} \max & \mathcal{F}(\underbrace{x,\cdots,x}_{2d}) \\ \st & \|x\|=1,
\end{array}
\end{equation}
where $\mathcal{F}$ is a $2d$-th order super-symmetric tensor.
In particular, problem~\eqref{prob:tensor-pca} can be equivalently written as
\begin{equation}\label{prob:tensor-pca1}
\begin{array}{ll} \max & \mathcal{F}\bullet \underbrace{x\otimes \cdots \otimes x}_{2d}\\ \st & \|x\|=1.
\end{array}
\end{equation}
Given any $2d$-th order super-symmetric tensor form $\mathcal{F}$, we call it {\it rank one}\/ if $\mathcal{F} = { \lambda}\underbrace{a\otimes \cdots \otimes a}_{2d}$ for some $a \in \RR^n$ and {{$\lambda \in \{ 1, -1 \}$}}. Moreover, the CP rank of $\mathcal{F}$ is defined as follows.
\begin{definition}\label{RankDef}
Suppose $\mathcal{F} \in \SI^{n^{2d}}$, the CP rank of $\mathcal{F}$ denoted by $\rank(\mathcal {F})$ is the smallest integer $r$ satisfying
\begin{equation}\label{formula:cp-rank-def}
\mathcal{F} = \sum_{i=1}^{r}\lambda_i\underbrace{a^i\otimes \cdots \otimes a^i}_{2d},
\end{equation}
where $a_i\in \RR^n, {\lambda_i \in \{ 1,-1 \}}$.
\end{definition}
{The idea of decomposing a tensor into an
(asymmetric) outer product of vectors was first introduced and studied by Hitchcock in 1927~\cite{hitchcock-27a,hitchcock-27b}. This concept of tensor-rank became popular after its rediscovery in the 1970's in the form of
CANDECOMP (canonical decomposition) by Carroll and Chang~\cite{carroll-70} and PARAFAC
(parallel factors) by Harshman~\cite{harshman-70}. Consequently, CANDECOMP and PARAFAC are further abbreviated as `CP' in the context of `CP rank' by many authors in the literature.

We remark that, the CP rank is theoretically associated with the complex number field, while in Definition~\ref{RankDef} decomposition \eqref{formula:cp-rank-def} is performed in the real domain. Though the choice of complex or real domain is immaterial for the matrices, it does make a difference in the tensor case~\cite{CGLM08}. Since we only focus on the real tensors here, throughout this paper we shall use the CP rank to denote the symmetric real rank
of a super-symmetric tensor.}

In the following, to simplify the notation, we denote $\K(n,d)=\left\{k=(k_1,\cdots, k_n)\in \Z_+^{n} \,\bigg{|}\,  \sum\limits_{j=1}^{n}k_j=d\right\}$ and
$$
\mathcal{X}_{1^{2k_1}2^{2k_2}\cdots n^{2k_n}}:=\mathcal{X}_{ \tiny \underbrace{1...1}_{2k_1}\underbrace{2...2}_{2k_2} \ldots \underbrace{n...n}_{2k_n} }.
$$
By letting $\mathcal{X}=\underbrace{x\otimes \cdots \otimes x}_{2d}$ we can further convert problem~\eqref{prob:tensor-pca1} into:
\begin{equation}\label{prob:tensor-pca2}
\begin{array}{ll} \max & \mathcal{F}\bullet \mathcal{X} \\ \st & \sum\limits_{k \in \K(n,d)}\frac{d!}{\prod_{j=1}^{n}{k_j!}}\mathcal{X}_{1^{2k_1}2^{2k_2}\cdots n^{2k_n}}  = 1,\\
& \mathcal{X} \in \SI^{n^{2d}},\;\rank(\mathcal{X})=1,
\end{array}
\end{equation}
where the first equality constraint is due to the fact that $ \sum\limits_{k \in \K(n,d)}\frac{d!}{\prod_{j=1}^{n}{k_j!}}\prod_{j=1}^{n}x_j^{2k_j}=\|x\|^{2d}=1$.

The difficulty of the above problem lies in the dealing of the rank constraint $\rank(\mathcal{X})=1$. Not only the rank function itself is difficult to deal with, but also determining the rank of a specific given tensor is already a difficult task, which is NP-hard in general \cite{H90}. To give an impression of the difficulty involved in computing tensor ranks, note that there is a particular $9 \times 9 \times 9$ tensor (cf.~\cite{K89}) whose rank is only known to be in between $18$ and $23$. One way to deal with the difficulty is to convert the tensor optimization problem~\eqref{prob:tensor-pca2} into a matrix optimization problem. A typical matricization technique is the so-called mode-$n$ matricization~\cite{KB09}. Roughly speaking, given a tensor $\mathcal {A} \in \RR^{n_1\times n_2\times \cdots\times n_m}$, its mode-$n$ matricization denoted by $A(n)$ is to arrange the $n$-th index of $\cal A$ to be the row index of the resulting matrix and merge all other indices of $\cal A$ as the column index of $A(n)$. {The precise definition of the mode-$n$ matricization is as follows.
\begin{definition}\label{def:mode-n-matricization}
For a given tensor $\mathcal{A} \in \RR^{n_1\times n_2\times\cdots \times n_m}$, the matrix $A(n)$ is the associated mode-$n$ matricization. In particular
$$
A(n)_{i_n,j} := \mathcal{A}_{i_1,i_2,\cdots , i_m}, \;\forall \; 1 \le i_k \le n_k, \; 1\le k \le m,
$$
where
\begin{equation}\label{formula:mode-n-index}j = 1+ \sum\limits_{{k=1} \atop{k\neq n} }^{m}(i_k-1)J_k,\mbox{ with }
J_k = \prod\limits_{{\ell=1}\atop{\ell \neq n}}^{k}n_\ell.
\end{equation}
\end{definition}
The so-called {\it $n$-rank} of $\cal A $ is defined by the vector $\left[\rank(A(1)),\rank(A(2)),\cdots, \rank(A(m))\right]$, where
its $n$-th component corresponds to the column rank of the mode-$n$ matrix $A(n)$.} The notion of $n$-rank has been widely used in the problems of tensor decomposition. Recently, Liu {\it et al.}~\cite{LMWY09} and Gandy {\it et al}.~\cite{Gandy-Recht-Yamada-2011} considered the low-$n$-rank tensor recovery problem, which were the first attempts to solve low-rank tensor optimization problems.
{Along this line, Tomioka {\it et al.}~\cite{TSHK2011} analyzed the statistical performance of nuclear norm relaxation of the tensor $n$-rank minimization problem.}
Chandrasekaran {\it et al.}~\cite{CRPW12} propose another interesting idea, in particular they directly apply convex relaxation to the tensor rank and obtain a new norm called tensor nuclear norm, which is numerically intractable. Thus, a further semidefinite representable relaxation is introduced. However, the authors did not provide any numerical results for this relaxation. Therefore, in the following we shall introduce a new scheme to unfold a tensor into a matrix, where we use half of the indices of tensor to form the row index of a matrix and use the other half as the column index. Most importantly, in the next section, we manage to establish some connection between the CP rank of the tensor and the rank of the resulting unfolding matrix.

\begin{definition}\label{def:matricization}
For a given { super-symmetric even-order tensor $\mathcal{F} \in \SI^{n^{2d}}$}, we define its square matricization, denoted by $\matr(\mathcal{F})\in \RR^{n^d\times n^d}$, as the following:
\[\matr(\mathcal{F})_{k \ell}:=\mathcal{F}_{i_1\cdots i_{d}i_{d+1} \cdots i_{2d}}, \quad 1\leq i_1,\ldots,i_d,i_{d+1},\ldots,i_{2d} \leq n,\]
where
\[k = \sum\limits_{j=1}^{d}(i_j-1)n^{d-j}+1,\mbox{ and }
\ell = \sum\limits_{j=d+1}^{2d}(i_j-1)n^{2d-j}+1.\]
\end{definition}
Similarly we introduce below the vectorization of a tensor.
\begin{definition}\label{def:vectorization}
The vectorization, $\vect(\mathcal{F})$, of tensor $\mathcal{F} \in \RR^{n^{m}}$ is defined as
\[\vect(\mathcal{F})_{k}:=\mathcal{F}_{i_1\cdots i_{m}},\]
where
\[k = \sum\limits_{j=1}^{m}(i_j-1)n^{m-j}+1, 1 \le i_1,\cdots, i_m \le n.\]
\end{definition}

In the same vein, we can convert a vector or a matrix with appropriate dimensions to a tensor. In other words, the inverse of the operators $\matr$ and $\vect$ can be defined in the same manner. In the following, we denote $X=\matr(\mathcal{X})$, and so
\[\tr(X) = \sum\limits_{\ell }X_{\ell, \ell}\;\mbox{with}\;\ell = \sum\limits_{j=1}^{d}(i_j-1)n^{d-j}+1.\]
If we assume $\mathcal{X}$ to be of rank one, then
\[\tr(X) = \sum\limits_{i_1,\cdots ,i_d }\mathcal{X}_{i_1\cdots i_d i_1 \cdots i_d}=\sum\limits_{i_1,\cdots ,i_d }\mathcal{X}_{i_1^2\cdots i_d^2}.\]
 In the above expression, $(i_1,\cdots ,i_d )$ is a subset of $(1,2,\ldots, n)$. Suppose that $j$ appears $k_j$ times in $(i_1,\cdots ,i_d )$ with $j=1,2,\ldots, n$ and $\sum\limits_{j=1}^{n}k_j=d$. Then for a fixed outcome $(k_1,k_2,\cdots, k_n)$, the total number of permutations $(i_1,\cdots ,i_d )$ to achieve such outcome is
\[
\left( d\atop{k_1} \right)\left( d-k_1\atop{k_2}\right) \left( d-k_1 -k_2\atop{k_3}\right)\cdots \left( d-k_1 -\cdots - k_{n-1}\atop{k_n}\right) = \frac{d!}{\prod_{j=1}^{n}{k_j!}}.
\]
Consequently,
\begin{equation} \label{trace-equality}
\tr(X)= \sum\limits_{i_1,\cdots ,i_d }\mathcal{X}_{i_1^2\cdots i_d^2} =\sum\limits_{k \in \K(n,d)}\frac{d!}{\prod_{j=1}^{n}{k_j!}}\mathcal{X}_{1^{2k_1}2^{2k_2}\cdots n^{2k_n}}.
\end{equation}
In light of the above discussion, if we further denote $F = \matr(\mathcal{F})$, then the objective in \eqref{prob:tensor-pca2} is $\mathcal{F} \bullet \mathcal{X} = \tr(FX)$, while the first constraint
$\sum\limits_{k \in \K(n,d)}\frac{d!}{\prod_{j=1}^{n}{k_j!}}\mathcal{X}_{1^{2k_1}2^{2k_2}\cdots n^{2k_n}}  = 1 \; \Longleftrightarrow \; \tr(X)=1$.
The hard constraint in \eqref{prob:tensor-pca2} is $\rank(\mathcal{X})=1$. It is straightforward to see that if $\mathcal{X}$ is of rank one, then by letting  $\mathcal{X}={\lambda}\underbrace{x\otimes \cdots \otimes x}_{2d}$ for some {$\lambda \in \{1,-1 \}$} and $\mathcal{Y}=\underbrace{x\otimes \cdots \otimes x}_{d}$, we have $\matr(\mathcal{X})={\lambda} \vect(\mathcal{Y})\vect(\mathcal{Y})^{\T}$, which is to say that matrix $\matr(\mathcal{X})$ is of rank one too. In the next section we shall continue to show that the other way around is also true.

\section{Equivalence Under the Rank-One Hypothesis}
{We first present some useful observations below.
\begin{lemma}\label{lemma:n-mode-rank1}
Suppose $\mathcal{A} \in \RR^{n^d}$ is an $n$ dimensional $d$-th order tensor and $\mathcal{A} \otimes \mathcal{A} \in \SI^{n^{2d}}$. Then we have:
\[\ba{ll} (i) & \mathcal{A}  \in \SI^{n^d}; \\
          (ii) & \mbox{the $n$-rank of}\;\mathcal{A}\;\mbox{is}\; [1,1,\cdots,1].\ea\]
\end{lemma}
\begin{proof}
We denote $\mathcal{F}=\mathcal{A} \otimes \mathcal{A} \in \SI^{n^{2d}}$. For any $d$-tuples $\{i_1,\cdots,i_d\}$, and one of its permutations $\{j_1,\cdots,j_d\} \in \pi(i_1,\cdots,i_d)$, it holds that
\begin{eqnarray*}
 \left(\mathcal{A} _{i_1,\cdots,i_d}-\mathcal{A} _{j_1,\cdots,j_d}\right)^2&=&\mathcal{A} _{i_1,\cdots,i_d}^2+\mathcal{A} _{j_1,\cdots,j_d}^2-2\mathcal{A} _{i_1,\cdots,i_d}\mathcal{A}_{j_1,\cdots,j_d}\\
 &=&\mathcal{F}_{i_1,\cdots,i_d,i_1,\cdots,i_d}+\mathcal{F}_{j_1,\cdots,j_d,j_1,\cdots,j_d}-2\mathcal{F}_{i_1,\cdots,i_d,j_1,\cdots,j_d}=0,
\end{eqnarray*}
where the last equality is due to the fact that $\FCal$ is super-symmetric.
Therefore, $\mathcal{A}$ is super-symmetric.

To prove the second statement, we first observe that for any two $d$-tuples $\{i_1,\cdots,i_d\}$ and $\{i^\prime_1,\cdots,i^\prime_d\}$, due to the super-symmetry of $\cal F$, we have
$$\mathcal{A}_{i_1,\cdots,i_d}\mathcal{A}_{i^\prime_1,\cdots,i^\prime_d}=\mathcal{F}_{i_1,\cdots,i_d,i^\prime_1,\cdots,i^\prime_d}=\mathcal{F}_{i^\prime_1,i_2,\cdots,i_d,i_1,i^\prime_2,\cdots,i^\prime_d}=\mathcal{A}_{i^\prime_1,i_2,\cdots,i_d}\mathcal{A}_{i_1,i^\prime_2,\cdots,i^\prime_d}.$$
Now consider the mode-$1$ unfolding, which is the matrix $A(1)$. For any two components $A(1)_{i_1,j}$ and $A(1)_{{i}^\prime_1,{j}^\prime}$ with $j = 1+ \sum\limits_{{k=2} }^{m}(i_k-1)J_k$, $j^\prime = 1+ \sum\limits_{{k=2} }^{m}(i^\prime_k-1)J_k$ and $J_k$ is defined in \eqref{formula:mode-n-index}, the equation above implies that
$$
A(1)_{i_1,j} A(1)_{{i}^\prime_1,{j}^\prime}=\mathcal{A}_{i_1,\cdots,i_d}\mathcal{A}_{i^\prime_1,\cdots,i^\prime_d}=\mathcal{A}_{i^\prime_1,i_2,\cdots,i_d}\mathcal{A}_{i_1,i^\prime_2,\cdots,i^\prime_d}=A(1)_{i^\prime_1,j} A(1)_{{i}_1,{j}^\prime}.
$$
Therefore, every $2 \times 2$ minor of matrix $A(1)$ is zero and so $A(1)$ is of rank one. 
Moreover, since
$\mathcal{A}$ is super-symmetric, the mode-unfolded matrices are all the same. Thus, we conclude that the $n$-rank of $\mathcal{A}$ is $[1,1,\cdots,1]$.
\end{proof}
The following lemma tells us if a super-symmetric tensor is of rank one in the sense of nonsymmetric CP, then the symmetric CP rank of the tensor is also one.
\begin{lemma}\label{lemma:rank-one}If a $d$-th order tensor $\mathcal{A} = a^1\otimes a^2 \otimes  \cdots \otimes a^d$ is super-symmetric, then we have $a^i =\pm a^1$ for $i=2,\ldots,d$ and
$\mathcal{A} =\lambda\, \underbrace{a^1\otimes a^1 \otimes  \cdots \otimes a^1}_{d}$ for some $\lambda = \pm 1$.
\end{lemma}
\begin{proof} Since $\mathcal{A}$ is super-symmetric, from Theorem 4.1 in \cite{CHLZ11}, we know that
$$
\max_{\|x\|=1} |\mathcal{A} (\underbrace{x,\cdots,x}_{d})|=\max_{\|x^i\|=1,i=1,\ldots,d} \mathcal{A} (x^1,\cdots,x^d)=\|a^1\|\times \|a^2\|\times \cdots \times \|a^d\| .
$$
So there must exist an $x^*$ with $\|x^*\|=1$ such that $ |(a^i)^{\top}x^*|=\|a^i\|$ for all $i$, which implies that $a^i =\pm a^1$ for $i=2,\ldots,d$, and thus the conclusion follows.
\end{proof}
We have the following proposition as the immediate consequence of the above lemmas.
\begin{proposition}\label{proposition:tensor-rank1}
Suppose $\mathcal{A} \in \RR^{n^d}$ is an $n$ dimensional $d$-th order tensor. The following two statements are equivalent:
\[\ba{ll} (i) & \mathcal{A}  \in \SI^{n^d}, \mbox{ and } \ \rank(\mathcal{A})=1; \\
          (ii) & \mathcal{A} \otimes \mathcal{A} \in \SI^{n^{2d}}.\ea\]
\end{proposition}
\begin{proof} We shall first show (i) $\Longrightarrow$ (ii). Suppose $\mathcal{A}  \in \SI^{n^d}$ with $\rank(\mathcal{A})=1$. Then there exists a vector $a \in \RR^{n}$ and a scaler $\lambda \in \{ 1, -1 \}$ such that
$\mathcal{A} = {\lambda} \underbrace{a\otimes a \otimes \cdots \otimes a}_{{d}}$. Consequently, $\mathcal{A} \otimes \mathcal{A} = \underbrace{a\otimes a \otimes \cdots \otimes a}_{2d} \in \SI^{n^{2d}} $.

Now we prove (ii) $\Longrightarrow$ (i). From Lemma~\ref{lemma:n-mode-rank1}, we know that $\cal A$ is super-symmetric and the $n$-rank of $\cal A$ is $[1,1,\cdots,1]$. It is well known that the $n$-rank of a tensor corresponds to the size of the {\it core tensor} associated with the {\it smallest exact Tucker decomposition}~\cite{KB09}. Consequently,
the $n$-rank of $\mathcal{A}$ is $[1,1,\cdots,1]$ means that the core tensor associated with the exact Tucker decomposition of $\mathcal{A}$ is a scalar, thus the nonsymmetric real CP rank of $\mathcal{A}$ is also one. Finally, due to Lemma~\ref{lemma:rank-one} and the fact that $\cal A$ is super-symmetric, we conclude that the symmetric CP rank of $\mathcal{A}$ is one, i.e.\ $\rank(\mathcal{A})=1$.
\end{proof}

}
Now we are ready to present the main result of this section.
\begin{theorem}\label{theorem:tensor-rank1}
Suppose $\mathcal{X} \in \SI^{n^{2d}}$ and $X = \matr(\mathcal{X})\in \RR^{n^d \times n^d}$. Then we have
$$\rank(\mathcal{X})=1\; \Longleftrightarrow\; \rank(X)=1.$$
\end{theorem}

\begin{proof}
As remarked earlier, that $\rank(\mathcal{X})=1\; \Longrightarrow\; \rank(X)=1$ is evident. To see this, suppose $\rank(\mathcal{X})=1$ and $\mathcal{X}=\underbrace{x\otimes \cdots \otimes x}_{2d}$ for some $x \in \RR^n$. By constructing $\mathcal{Y}=\underbrace{x\otimes \cdots \otimes x}_{d}$, we have $X=\matr(\mathcal{X})=\vect(\mathcal{Y})\vect(\mathcal{Y})^{\T}$, which leads to $\rank(X)=1$.

To prove the other implication, suppose that we have $\mathcal{X} \in \SI^{n^{2d}}$ and $\matr(\mathcal{X})$ is of rank one, i.e. $\matr(\mathcal{X})=yy^{\T}$ for some vector $y\in\RR^{n^d}$. Then $\mathcal{X} = \vect^{-1}(y)\otimes \vect^{-1}(y)$, which combined with Proposition~\ref{proposition:tensor-rank1} implies $\vect^{-1}(y)$ is supper-symmetric and of rank one. Thus there exists $x \in \RR^n$ such that $\vect^{-1}(y) = \underbrace{x\otimes \cdots \otimes x}_{d}$ and $\mathcal{X}=\underbrace{x\otimes \cdots \otimes x}_{2d}$.
\end{proof}

\section{A Nuclear Norm Penalty Approach}

According to Theorem \ref{theorem:tensor-rank1}, we know that a super-symmetric tensor is of rank one, if and only if its matrix correspondence obtained via the matricization operation defined in Definition \ref{def:matricization}, is also of rank one. As a result, we can reformulate Problem \eqref{prob:tensor-pca2} equivalently as the following matrix optimization problem:
\begin{equation}\label{prob:matrix-pca}
\begin{array}{ll} \max & \tr(F X) \\
                    \st & \tr(X) =1, \ \matr^{-1}(X) \in \SI^{n^{2d}}, \\
                        & X\in\SI^{n^d \times n^d}, \ \rank(X)=1,
\end{array}
\end{equation}
where $X=\matr(\mathcal{X})$, $F = \matr(\mathcal{F})$, and $\SI^{n^d \times n^d}$ denotes the set of $n^d\times n^d$ symmetric matrices.
Note that the constraints $\matr^{-1}(X) \in \SI^{n^{2d}}$ requires the tensor correspondence of $X$ to be super-symmetric, which essentially correspond to $O(n^{2d})$ linear equality constraints. The rank constraint $\rank(X)=1$ makes the problem intractable. In fact, Problem \eqref{prob:matrix-pca} is NP-hard in general, due to its equivalence to problem \eqref{prob:tensor-pca}.

There have been a large amount of work that deal with the low-rank matrix optimization problems. Research in this area was mainly ignited by the recent emergence of compressed sensing \cite{Candes-Romberg-Tao-2006,Donoho-06}, matrix rank minimization and low-rank matrix completion problems \cite{Recht-Fazel-Parrilo-2007,Candes-Recht-2008,Candes-Tao-2009}. The matrix rank minimization seeks a matrix with the lowest rank satisfying some linear constraints, i.e.,
\be\label{prob:matrix-rank-min} \min_{X\in \RR^{n_1\times n_2}} \ \rank(X), \ \st, \ \mathcal{C}(X)=b,\ee where $b\in\RR^p$ and $\mathcal{C}:\RR^{n_1\times n_2}\rightarrow \RR^p$ is a linear operator. The results in \cite{Recht-Fazel-Parrilo-2007,Candes-Recht-2008,Candes-Tao-2009} show that under certain randomness hypothesis on the linear operator $\mathcal{C}$, with high probability the NP-hard problem \eqref{prob:matrix-rank-min} is equivalent to the following nuclear norm minimization problem, which is a convex programming problem:
\be\label{prob:nuclear-norm-min} \min_{X\in \RR^{n_1\times n_2}} \ \|X\|_*, \ \st, \ \mathcal{C}(X)=b. \ee
In other words, the optimal solution to the convex problem \eqref{prob:nuclear-norm-min} is also the optimal solution to the original NP-hard problem \eqref{prob:matrix-rank-min}.

Motivated by the convex nuclear norm relaxation, one way to deal with the rank constraint in \eqref{prob:matrix-pca} is to introduce the nuclear norm term of $X$, which penalizes high-ranked $X$'s, in the objective function. This yields the following convex optimization formulation:
\begin{equation}\label{prob:nuclear-penalty}
\begin{array}{ll} \max & \tr(F X) - \rho \| X \|_*\\
                              \st & \tr(X) =1, \ \matr^{-1}(X) \in \SI^{n^{2d}}, \\
                                  & X\in\SI^{n^d\times n^d},
\end{array}
\end{equation}
where $\rho>0$ is a penalty parameter.
It is easy to see that if the optimal solution of \eqref{prob:nuclear-penalty} (denoted by $\tilde{X}$) is of rank one, then $\|\tilde{X}\|_* = \tr(\tilde{X})=1$, which is a constant. In this case, the term $-\rho\|X\|_*$ added to the objective function is a constant, which leads to the fact the solution is also optimal with the constraint that $X$ is rank-one. In fact, Problem \eqref{prob:nuclear-penalty} is the convex relaxation of the following problem
\begin{equation*}
\begin{array}{ll} \max & \tr(F X) - \rho \| X \|_* \\
                              \st & \tr(X) =1, \ \matr^{-1}(X) \in \SI^{n^{2d}}, \\
                                  & X\in\SI^{n^d\times n^d},\ \rank(X)=1,
\end{array}
\end{equation*}
which is equivalent to the original problem ~\eqref{prob:matrix-pca} since $\rho \| X \|_* = \rho \tr(X)=\rho$.

After solving the convex optimization problem \eqref{prob:nuclear-penalty} and obtaining the optimal solution $\tilde{X}$, if $\rank(\tilde{X})=1$, we can find $\tilde{x}$ such that $\matr^{-1}(\tilde{X})=\underbrace{\tilde{x}\otimes \cdots \otimes \tilde{x}}_{2d}$, according to Theorem \ref{theorem:tensor-rank1}. In this case, $\tilde{x}$ is the optimal solution to Problem \eqref{prob:tensor-pca}. The original tensor PCA problem, or the Z-eigenvalue problem \eqref{prob:tensor-pca}, is thus solved to optimality.

Interestingly, we found from our extensive numerical tests that the optimal solution to Problem \eqref{prob:nuclear-penalty} is a rank-one matrix almost all the time. In the following, we will show this interesting phenomenon by some concrete examples. The first example is taken from \cite{KR02}.
\begin{example}We consider a super-symmetric tensor $\mathcal{F} \in \SI^{3^{4}}$ defined by
\begin{eqnarray*}
\mathcal{F}_{1111}=0.2883,\quad \mathcal{F}_{1112}=-0.0031,\quad \mathcal{F}_{1113}=0.1973,\quad \mathcal{F}_{1122}=-0.2485, \quad \mathcal{F}_{1123}=-0.2939,\\
\mathcal{F}_{1133}=0.3847,\quad \mathcal{F}_{1222}=0.2972,\quad \quad \mathcal{F}_{1223}=0.1862,\quad \mathcal{F}_{1233}=0.0919, \quad \; \mathcal{F}_{1333}=-0.3619,\\
\mathcal{F}_{2222}=0.1241,\quad \mathcal{F}_{2223}=-0.3420,\quad \; \mathcal{F}_{2233}=0.2127,\quad \mathcal{F}_{2333}=0.2727, \quad \; \mathcal{F}_{3333}=-0.3054.
\end{eqnarray*}
We want to compute the largest Z-eigenvalue of $\mathcal{F}$.
\end{example}

Since the size of this tensor is small, we used CVX \cite{cvx} to solve Problem~\eqref{prob:nuclear-penalty} with $F = \matr(\mathcal{F})$ and $\rho = 10$. It turned out that CVX produced a rank-one solution $\tilde{X} = aa^{\top}\in\RR^{3^2\times 3^2}$, where
\[ a =(0.4451,0.1649,-0.4688,0.1649,0.0611,-0.1737,-0.4688,-0.1737,0.4938)^{\top}.\]
Thus we get the matrix correspondence of $a$ by reshaping $a$ into a square matrix $A$:
\[
A=[a(1:3),a(4:6),a(7:9)]= \begin{bmatrix} 0.4451 & 0.1649 &-0.4688\\
                                                  0.1649& 0.0611 &-0.1737\\
                                                  -0.4688&-0.1737&0.4938\end{bmatrix}.
\]
It is easy to check that $A$ is a rank-one matrix with the nonzero eigenvalue being 1. This further confirms our theory on the rank-one equivalence, i.e., Theorem \ref{theorem:tensor-rank1}. The eigenvector that corresponds to the nonzero eigenvalue of $A$ is given by
\[\tilde{x} = (-0.6671,-0.2472,0.7027)^{\top},\]
which is the optimal solution to Problem \eqref{prob:tensor-pca}.

The next example is from a real Magnetic Resonance Imaging (MRI) application studied by Ghosh \etal in \cite{GTDPMR08}. In \cite{GTDPMR08}, Ghosh \etal studied a fiber detection problem in
diffusion Magnetic Resonance Imaging (MRI), where they tried to extract the geometric characteristics from an antipodally symmetric spherical function (ASSF), which can be described
equivalently in the homogeneous polynomial basis constrained to the sphere. They showed that it is possible to extract the maxima and minima of an ASSF by computing the stationary points of a problem in the form of~\eqref{prob:tensor-pca} with $d=2$ and $n=4$.
\begin{example}The objective function $\mathcal{F}({x,x,x,x})$ in this example is given by
\[\begin{array}{rrrrrrrrrr}
&0.74694 x_1^4& -& 0.435103 x_1^3x_2& +& 0.454945 x_1^2x_2^2& +& 0.0657818 x_1x_2^3 & +& x_2^4 \\
+& 0.37089 x_1^3x_3 & -& 0.29883 x_1^2x_2x_3 & - & 0.795157 x_1x_2^2x_3 & + & 0.139751 x_2^3x_3  & + & 1.24733 x_1^2x_3^2\\
+& 0.714359 x_1x_2x_3^2 & +& 0.316264 x_2^2x_3^2 & -& 0.397391 x_1x_3^3  & -& 0.405544 x_2x_3^3& + & 0.794869 x_3^4.
\end{array} \]
\end{example}
Again, we used CVX to solve problem~\eqref{prob:nuclear-penalty} with $F = \matr(\mathcal{F})$ and $\rho = 10$, and a rank-one solution was found with $\tilde{X}= aa^{\T}$, with
\[a=(0.0001, 0.0116, 0.0004, 0.0116, 0.9984, 0.0382, 0.0004, 0.0382, 0.0015)^{\T}.\]
By reshaping vector $a$, we get the following expression of matrix $A$:
\[ A = [a(1:3),a(4:6),a(7:9)] = \begin{bmatrix}        0.0001      & 0.0116 &  0.0004\\
                                                       0.0116      & 0.9984 &  0.0382\\
                                                       0.0004      & 0.0382 &  0.0015\end{bmatrix}.\]
It is easy to check that $A$ is a rank-one matrix with 1 being the nonzero eigenvalue. The eigenvector corresponding to the nonzero eigenvalue of $A$ is given by \[\tilde{x} = (0.0116,0.9992,0.0382)^{\T},\] which is also the optimal solution to the original problem \eqref{prob:tensor-pca}.

Henceforth we conduct some numerical tests on randomly generated examples. We construct $4$-th order tensor $\cal T$ with its components drawn randomly from i.i.d.\ standard normal distribution. The super-symmetric tensor $\cal F$ in the tensor PCA problem is obtained by symmetrizing $\cal T$.
All the numerical experiments in this paper were conducted on an Intel Core i5-2520M 2.5GHz computer with 4GB of RAM, and all the default settings of {Matlab 2012b and CVX 1.22} were used for all the tests. We choose $d=2$ and the dimension of $\cal F$ in the tensor PCA problem from $n=3$ to $n=9$. We choose $\rho=10$. For each $n$, we tested $100$ random instances. In Table \ref{tab:rank1-frequency-NNP}, we report the number of instances that produced rank-one solutions. We also report the average CPU time (in seconds) using CVX to solve the problems.

\begin{table}[htb]
\centering
\begin{tabular}{c|r|r}
\hline
 $n$ & $\rank$-1 & CPU \\\hline
 3& 100 & 0.21\\
 4& 100 & 0.56\\
 5& 100 & 1.31\\
 6& 100 & 6.16\\
 7& 100 & 47.84\\
 8& 100 & 166.61\\
 9& 100 & 703.82\\
\hline
\end{tabular}
\caption{Frequency of nuclear norm penalty problem~\eqref{prob:nuclear-penalty} having a rank-one solution}
\label{tab:rank1-frequency-NNP}
\end{table}

Table \ref{tab:rank1-frequency-NNP} shows that for these randomly created tensor PCA problems, the nuclear norm penalty problem \eqref{prob:nuclear-penalty} {\it always} gives a rank-one solution, and thus {\it always} solves the original problem \eqref{prob:tensor-pca} to optimality.



\section{Semidefinite Programming Relaxation}

In this section, we study another convex relaxation for Problem \eqref{prob:matrix-pca}. Note that the constraint
\[X\in\SI^{n^d\times n^d}, \rank(X)=1\]
in \eqref{prob:matrix-pca} actually implies that $X$ is positive semidefinite. To get a tractable convex problem, we drop the $\rank$ constraint and impose a semidefinite constraint to \eqref{prob:matrix-pca} and consider the following SDP relaxation:
\begin{equation}\label{prob:matrix-pca-SDR}
\begin{array}{lll} (SDR)&\max & \tr(F X) \\
                        &\st & \tr(X) =1,\\
                        && \matr^{-1}(X) \in \SI^{n^{2d}},\;X \succeq 0.
\end{array}
\end{equation}
Remark that replacing the rank-one constraint by SDP constraint is by now a common and standard practice; 
see, e.g., \cite{Alizadeh93interiorpoint,GoemansWilliamson1995,VandenbergheBoyd1996}.
Next theorem shows that the SDP relaxation \eqref{prob:matrix-pca-SDR} is actually closely related to the nuclear norm penalty problem \eqref{prob:nuclear-penalty}.

\begin{theorem}\label{thoerem:SDR-NuclearNormPenalty}
Let $X^*_{SDR}$ and $X^*_{PNP}(\rho)$ be the optimal solutions of problems~\eqref{prob:matrix-pca-SDR} and~\eqref{prob:nuclear-penalty} respectively. Suppose $Eig^+(X)$ and $Eig^-(X)$ are the summations of nonnegative eigenvalues and negative eigenvalues of $X$ respectively, i.e.,
\[Eig^+(X) := \sum_{i:\: \lambda_i(X)\geq 0} \lambda_i(X), \quad Eig^-(X) := \sum_{i:\: \lambda_i(X)< 0} \lambda_i(X).\]
It holds that
\[{2(\rho - v)}\,\left|Eig^-(X^*_{PNP}(\rho))\right| \le v - F_0,\]
where $F_0 := \max\limits_{1\le i \le n}\mathcal{F}_{i^{2d}}$ and $v$ is the optimal value of the following optimization problem
\begin{equation}\label{prob:NuclearNormConstraint}
\begin{array}{lll} &\max & \tr(F X) \\
                        &\st & \|X\|_* =1,\\
                        && X \in \SI^{n^d \times n^d}.
\end{array}
\end{equation}
As a result, $\lim\limits_{\rho \rightarrow + \infty}\tr(FX^*_{PNP}(\rho))=\tr(FX^*_{SDR})$.
\end{theorem}

\begin{proof} Observe that $\matr(\underbrace{e^i\otimes \cdots \otimes e^i}_{2d})$, where $e^i$ is the $i$-th unit vector, 
is a feasible solution for problem~\eqref{prob:nuclear-penalty} with objective value $\mathcal{F}_{i^{2d}} - \rho$ for all $1 \le i \le n$. Moreover, by denoting $r(\rho) = \left|Eig^-(X^*_{PNP}(\rho))\right|$, we have
\begin{eqnarray*}
\|X^*_{PNP}(\rho)\|_* &=& Eig^+(X^*_{PNP}(\rho)) + \left|Eig^-(X^*_{PNP}(\rho))\right|\\
&=&\left( Eig^+(X^*_{PNP}(\rho)) + Eig^-(X^*_{PNP}(\rho)) \right)+ 2\left|Eig^-(X^*_{PNP}(\rho))\right|\\
&=&1+2r(\rho).
\end{eqnarray*}
Since $X^*_{PNP}(\rho)$ is optimal to problem~\eqref{prob:nuclear-penalty}, we have
\begin{equation}\label{comparing-objective}
\tr(FX^*_{PNP}(\rho)) - \rho\,(1+2r(\rho)) \ge \max\limits_{1\le i \le n}\mathcal{F}_{i^{2d}} - \rho \ge F_0 - \rho.
\end{equation}
Moreover, since $X^*_{PNP}(\rho)/\|X^*_{PNP}(\rho) \|_*$ is feasible to problem~\eqref{prob:NuclearNormConstraint}, we have
\be\label{comparing-objective-eq-before} \tr(FX^*_{PNP}(\rho)) \le \|X^*_{PNP}(\rho) \|_*\,v =(1+2r(\rho))\,v.\ee
Combining \eqref{comparing-objective-eq-before} and \eqref{comparing-objective} yields
\begin{equation}\label{r-rho}
{2(\rho - v)}\,r(\rho) \le v - F_0.
\end{equation} 
Notice that $\|X \|_* = 1$ implies $\|X\|_\infty$ is bounded for all feasible $X \in \SI^{n^d \times n^d}$, where $\|X\|_\infty$ denotes the largest entry of $X$ in magnitude. Thus the set $\{ X^*_{PNP}(\rho) \mid \rho > 0\}$ is bounded.
Let $X^*_{PNP}$ be one cluster point of sequence $\{ X^*_{PNP}(\rho) \mid \rho > 0 \}$. By taking the limit $\rho \rightarrow +\infty$ in~\eqref{r-rho}, we have $r(\rho)\rightarrow 0$ and thus $X^*_{PNP} \succeq 0$. Consequently, $X^*_{PNP}$ is a feasible solution to problem~\eqref{prob:matrix-pca-SDR} and $\tr(FX^*_{SDR}) \ge \tr(FX^*_{PNP})$. On the other hand, it is easy to check that for any $0 < \rho_1 < \rho_2$,
$$
\tr(FX^*_{SDR}) \le \tr(FX^*_{PNP}(\rho_2))  \le \tr(FX^*_{PNP}(\rho_1)),
$$
which implies $\tr(FX^*_{SDR}) \le \tr(FX^*_{PNP})$. Therefore, $\lim\limits_{\rho \rightarrow + \infty}\tr(FX^*_{PNP}(\rho)) = \tr(FX^*_{PNP})=\tr(FX^*_{SDR})$.
\end{proof}

Theorem \eqref{thoerem:SDR-NuclearNormPenalty} shows that when $\rho$ goes to infinity in \eqref{prob:nuclear-penalty}, the optimal solution of the nuclear norm penalty problem \eqref{prob:nuclear-penalty} converges to the optimal solution of the SDP relaxation \eqref{prob:matrix-pca-SDR}. As we have shown in Table \ref{tab:rank1-frequency-NNP} that the nuclear norm penalty problem \eqref{prob:nuclear-penalty} returns rank-one solutions for all the randomly created tensor PCA problems that we tested, it is expected that the SDP relaxation \eqref{prob:matrix-pca-SDR} will also {be likely to} give rank-one solutions. In fact, this is indeed the case as shown through the numerical results in Table \ref{tab:rank1-frequency-SDR}. As in Table \ref{tab:rank1-frequency-NNP}, we tested $100$ random instances for each $n$. In Table \ref{tab:rank1-frequency-SDR}, we report the number of instances that produced rank-one solutions for $d=2$. We also report the average CPU time (in seconds) using CVX to solve the problems. As we see from Table \ref{tab:rank1-frequency-SDR}, for these randomly created tensor PCA problems, the SDP relaxation \eqref{prob:matrix-pca-SDR} {\it always} gives a rank-one solution, and thus {\it always} solves the original problem \eqref{prob:tensor-pca} to optimality.

\begin{table}[htb]\small
\centering
\begin{tabular}{c|r|r}\hline
$n$ & $\rank$-1 & CPU \\\hline
 3& 100 & 0.14\\
 4& 100 & 0.25\\
 5& 100 & 0.55\\
 6& 100 & 1.16\\
 7& 100 & 2.37\\
 8& 100 & 4.82\\
 9& 100 & 8.89 \\\hline
\end{tabular}
\caption{Frequency of SDP relaxation \eqref{prob:matrix-pca-SDR} having a rank-one solution}
\label{tab:rank1-frequency-SDR}
\end{table}


\section{Alternating Direction Method of Multipliers}

The computational times reported in Tables \ref{tab:rank1-frequency-NNP} and \ref{tab:rank1-frequency-SDR} suggest that it can be time-consuming to solve the convex problems \eqref{prob:nuclear-penalty} and \eqref{prob:matrix-pca-SDR} when the problem size is large (especially for the nuclear norm penalty problem \eqref{prob:nuclear-penalty}). In this section, we propose an alternating direction method of multipliers (ADMM) for solving \eqref{prob:nuclear-penalty} and \eqref{prob:matrix-pca-SDR} that fully takes advantage of the structures. ADMM is closely related to some operator-splitting methods, known as Douglas-Rachford and Peaceman-Rachford methods, that were proposed in 1950s for solving variational problems arising from PDEs (see \cite{Douglas-Rachford-56,Peaceman-Rachford-55}). These operator-splitting methods were extensively studied later in the literature for finding the zeros of the sum of monotone operators and for solving convex optimization problems (see \cite{Lions-Mercier-79,Fortin-Glowinski-1983,Glowinski-LeTallec-89,Eckstein-thesis-89,Eckstein-Bertsekas-1992}). The ADMM we will study in this section was shown to be equivalent to the Douglas-Rachford operator-splitting method applied to convex optimization problem (see \cite{Gabay-83}). ADMM was revisited recently as it was found to be very efficient for many sparse and low-rank optimization problems arising from the recent emergence of compressed sensing \cite{Yang-Zhang-2009}, compressive imaging \cite{Wang-Yang-Yin-Zhang-2008,Goldstein-Osher-08}, robust PCA \cite{Tao-Yuan-SPCP-2011}, sparse inverse covariance selection \cite{Yuan-2009,Scheinberg-Ma-Goldfarb-NIPS-2010}, sparse PCA \cite{Ma-SPCA-2011-submit} and SDP \cite{Wen-Goldfarb-Yin-2009} etc. For a more complete discussion and list of references on ADMM, we refer to the recent survey paper by Boyd \etal \cite{Boyd-etal-ADM-survey-2011} and the references therein.

Generally speaking, ADMM solves the following convex optimization problem,
\be\label{prob:sum-2}\ba{ll} \min_{x\in\RR^n,y\in\RR^p} & f(x) + g(y) \\
                               \st      & Ax + By = b \\
                                        & x\in {\CCal}, \ y\in {\DCal}, \ea\ee
where $f$ and $g$ are convex functions, $A\in\RR^{m\times n}$, $B\in\RR^{m\times p}$, $b\in\RR^m$, $\CCal$ and $\DCal$ are some simple convex sets. A typical iteration of ADMM for solving \eqref{prob:sum-2} can be described as follows:
\be\label{alg:ADMM-sum-2}\left\{\ba{lll} x^{k+1} & := & \argmin_{x\in\CCal} \ \LCal_\mu(x,y^k;\lambda^k) \\
                                         y^{k+1} & := & \argmin_{y\in\DCal} \ \LCal_\mu(x^{k+1},y;\lambda^k) \\
                                         \lambda^{k+1} & := & \lambda^k - (Ax^{k+1}+By^{k+1}-b)/\mu, \ea\right. \ee
where the augmented Lagrangian function $\LCal_\mu(x,y;\lambda)$ is defined as
\[\LCal_\mu(x,y;\lambda) := f(x) + g(y) - \langle \lambda, Ax+By-b \rangle + \frac{1}{2\mu}\|Ax+By-b\|^2,\]
with $\lambda$ being the Lagrange multiplier and $\mu>0$ a penalty parameter.
The following theorem gives the global convergence of \eqref{alg:ADMM-sum-2} for solving \eqref{prob:sum-2}, and this has been well studied in the literature (see, e.g., \cite{Fortin-Glowinski-1983,Eckstein-thesis-89}).
\begin{theorem}\label{the:convergence-admm-generic}
Assume both A and B are of full column rank, the sequence $\{(x^k,y^k,\lambda^k)\}$ generated by \eqref{alg:ADMM-sum-2} globally converges to a pair of primal and dual optimal solutions $(x^*,y^*)$ and $\lambda^*$ of \eqref{prob:sum-2} from any starting point.
\end{theorem}

Because both the nuclear norm penalty problem \eqref{prob:nuclear-penalty} and SDP relaxation \eqref{prob:matrix-pca-SDR} can be rewritten in the form of \eqref{prob:sum-2}, we can apply ADMM to solve them.

\subsection{ADMM for Nuclear Norm Penalty Problem \eqref{prob:nuclear-penalty}}
Note that the nuclear norm penalty problem \eqref{prob:nuclear-penalty} can be rewritten equivalently as
\begin{equation}\label{prob:nuclear-penalty-XY}
\begin{array}{lll} \min &  - \tr(F Y) + \rho \| Y \|_* \\
                    \st &  X - Y = 0, \\
                        &  X \in {\CCal},
\end{array}
\end{equation}
where $\CCal:=\{X\in\SI^{n^d\times n^d}\mid \tr(X) =1, \ \matr^{-1}(X) \in \SI^{n^{2d}}\}$. A typical iteration of ADMM for solving \eqref{prob:nuclear-penalty-XY} can be described as
\be\label{alg:ADMM-nuclear-norm}\left\{\ba{lll} X^{k+1} & := & \argmin_{X \in \CCal} - \tr(F Y^k) + \rho \|Y^k\|_* - \langle \Lambda^k, X-Y^k \rangle + \frac{1}{2\mu}\|X-Y^k\|_F^2 \\ Y^{k+1} & := & \argmin \ -\tr(F Y) + \rho \|Y\|_* - \langle \Lambda^k, X^{k+1}-Y \rangle + \frac{1}{2\mu}\|X^{k+1}-Y\|_F^2 \\
\Lambda^{k+1} & := & \Lambda^k - (X^{k+1}-Y^{k+1})/\mu, \ea\right.\ee
where $\Lambda$ is the Lagrange multiplier associated with the equality constraint in \eqref{prob:nuclear-penalty-XY} and $\mu>0$ is a penalty parameter. Following Theorem \ref{the:convergence-admm-generic}, we know that the sequence $\{(X^k,Y^k,\Lambda^k)\}$ generated by \eqref{alg:ADMM-nuclear-norm} globally converges to a pair of primal and dual optimal solutions $(X^*,Y^*)$ and $\Lambda^*$ of \eqref{prob:nuclear-penalty-XY} from any starting point.

Next we show that the two subproblems in \eqref{alg:ADMM-nuclear-norm} are both easy to solve. The first subproblem in \eqref{alg:ADMM-nuclear-norm} can be equivalently written as
\be\label{alg:ADMM-nuclear-norm-sub-X} X^{k+1} := \argmin_{X\in\CCal} \half\|X-(Y^k+\mu\Lambda^k)\|_F^2, \ee
i.e., the solution of the first subproblem in \eqref{alg:ADMM-nuclear-norm} corresponds to the projection of $Y^k+\mu\Lambda^k$ onto convex set $\CCal$. We will elaborate how to compute this projection in Section \ref{sec:projection}.

The second subproblem in \eqref{alg:ADMM-nuclear-norm} can be reduced to:
\begin{equation}\label{alg:ADMM-subY-reduced} Y^{k+1} := \argmin_Y \quad\mu\rho\|Y\|_*+\half\|Y-(X^{k+1}-\mu(\Lambda^k-F))\|_F^2.\end{equation}
This problem is known to have a closed-form solution that is given by the following so-called matrix shrinkage operation (see, e.g., \cite{Ma-Goldfarb-Chen-2008}):
\begin{equation*} Y^{k+1} := U\Diag(\max\{\sigma-\mu\rho,0\})V^{\T},\end{equation*} where $U\Diag(\sigma)V^{\T}$ is the singular value decomposition of matrix
$X^{k+1}-\mu(\Lambda^k-F)$.

\subsection{The Projection}\label{sec:projection}
In this subsection, we study how to solve \eqref{alg:ADMM-nuclear-norm-sub-X}, i.e., how to compute the following projection for any given matrix $Z\in\SI^{n^d\times n^d}$:
\begin{equation}\label{projection-matrix-form}
\begin{array}{ll} \min & \|X - Z\|_{F}^2 \\ \st & \tr(X) =1,\\
                                & \matr^{-1}(X) \in \SI^{n^{2d}}.
\end{array}
\end{equation}
For the sake of discussion, in the following we consider the equivalent tensor representation of \eqref{projection-matrix-form}:
\begin{equation}\label{projection-tensor}
\begin{array}{ll} \min & \| \mathcal{X} - \mathcal{Z}\|_{F}^2 \\ \st & \sum\limits_{k \in \K(n,d)}\frac{d!}{\prod_{j=1}^{n}{k_j!}}\mathcal{X}_{1^{2k_1}2^{2k_2}\cdots n^{2k_n}} =1,\\
                                & \mathcal{X} \in \SI^{n^{2d}},
\end{array}
\end{equation}
where $\mathcal{X} = \matr^{-1}(X)$, $\mathcal{Z} = \matr^{-1}(Z)$, and the equality constraint is due to~\eqref{trace-equality}. Now we denote index set
$$
\mathbf{I} = \left\{ (i_1\cdots i_{2d}) \in \pi (1^{2k_1}\cdots n^{2k_n} ) \; \big{|}\; k=(k_1,\cdots,k_n) \in \K(n,d) \right\}.
$$
Then the first-order optimality conditions of Problem \eqref{projection-tensor} imply
\begin{equation*}
\left\{
\begin{array}{ll}
2 \left( |\pi(i_1\cdots i_{2d})|\,\mathcal{X}_{i_1\cdots i_{2d}} - \sum\limits_{j_1\cdots j_{2d} \in \pi(i_1\cdots i_{2d})}\mathcal{Z}_{j_1\cdots j_{2d}} \right) =0, &\mbox{if}\; (i_1\cdots i_{2d}) \not\in \mathbf{I}, \\
2 \left( \frac{(2d)!}{\prod_{j=1}^{n}{(2k_j)!}}\mathcal{X}_{1^{2k_1}\cdots n^{2k_n}} - \sum\limits_{j_1\cdots j_{2d} \in \pi(1^{2k_1}\cdots n^{2k_n})}\mathcal{Z}_{j_1\cdots j_{2d}} \right) - \lambda\, \frac{(d)!}{\prod_{j=1}^{n}{(k_j)!}} =0,  &\mbox{otherwise}.
\end{array}
\right.
\end{equation*}
Denote $\hat {\mathcal{Z}}$ to be the super-symmetric counterpart of tensor $\mathcal{Z}$, i.e.
$$
\hat {\mathcal{Z}}_{i_1\cdots i_{2d}} =  \sum\limits_{j_1\cdots j_{2d} \in \pi(i_1\cdots i_{2d})}\frac{\mathcal{Z}_{j_1\cdots j_{2d}}}{| \pi(i_1\cdots i_{2d})|}
$$
and $\alpha(k,d):= \big{(}\frac{(d)!}{\prod_{j=1}^{n}{(k_j)!}} \big{)}/ \big{(} \frac{(2d)!}{\prod_{j=1}^{n}{(2k_j)!}} \big{)}$. Then due to the first-order optimality conditions of \eqref{projection-tensor}, the optimal solution $\mathcal{X}^*$ of Problem \eqref{projection-tensor} satisfies
\begin{equation}\label{projection-summing-eq}
\left\{
\begin{array}{llll}
\mathcal{X}_{i_1\cdots i_{2d}}^* & = & \hat{\mathcal{Z}}_{i_1\cdots i_{2d}}, & \mbox{ if } (i_1\cdots i_{2d}) \not\in \mathbf{I}, \\
\mathcal{X}_{1^{2k_1}\cdots n^{2k_n}}^* & = & \frac{\lambda}{2}\, \alpha(k,d) + \hat{\mathcal{Z}}_{1^{2k_1}\cdots n^{2k_n}}, & \mbox{ otherwise }.
\end{array}
\right.
\end{equation}
Multiplying the second equality of \eqref{projection-summing-eq} by $\frac{(d)!}{\prod_{j=1}^{n}{(k_j)!}}$ and summing the resulting equality over all $k=(k_1,\cdots, k_n)$ yield
\[
\sum\limits_{k \in \K(n,d)}  \frac{(d)!}{\prod_{j=1}^{n}{(k_j)!}} \mathcal{X}_{1^{2k_1}\cdots n^{2k_n}}^* = \frac{\lambda}{2} \sum\limits_{k \in \K(n,d)} \frac{(d)!}{\prod_{j=1}^{n}{(k_j)!}}\alpha(k,d) + \sum\limits_{k \in \K(n,d)}  \frac{(d)!}{\prod_{j=1}^{n}{(k_j)!}}\hat{\mathcal{Z}}_{1^{2k_1}\cdots n^{2k_n}} .
\]
It remains to determine $\lambda$.
Noticing that $\mathcal{X}^*$ is a feasible solution for problem~\eqref{projection-tensor}, we have $\sum\limits_{k \in \K(n,d)}  \frac{(d)!}{\prod_{j=1}^{n}{(k_j)!}} \mathcal{X}_{1^{2k_1}\cdots n^{2k_n}}^*=1$. As a result,
\[
\lambda = 2\bigg{(}1 - \sum\limits_{k \in \K(n,d)}  \frac{(d)!}{\prod_{j=1}^{n}{(k_j)!}}\hat{\mathcal{Z}}_{1^{2k_1}\cdots n^{2k_n}} \bigg{)} \bigg{/} \sum\limits_{k \in \K(n,d)} \frac{(d)!}{\prod_{j=1}^{n}{(k_j)!}}\alpha(k,d),
\]
and thus we derived $\mathcal{X}^*$ and ${X}^* = \matr(\mathcal{X}^*)$ as the desired optimal solution for~\eqref{projection-matrix-form}.

\subsection{ADMM for SDP Relaxation \eqref{prob:matrix-pca-SDR}}
Note that the SDP relaxation problem \eqref{prob:matrix-pca-SDR} can be formulated as
\begin{equation}\label{prob:SDR-XY}
\begin{array}{ll} \min &  - \tr(F Y)\\
                    \st &  \tr(X) =1, \quad \matr^{-1}(X) \in \SI^{n^{2d}} \\
                        &  X - Y = 0, \quad Y \succeq 0.
\end{array}
\end{equation}
A typical iteration of ADMM for solving \eqref{prob:SDR-XY} is
\be\label{alg:ADMM-SDR}\left\{\ba{lll} X^{k+1} & := & \argmin_{X \in \CCal} -\tr(F Y^k) - \langle \Lambda^k, X-Y^k \rangle + \frac{1}{2\mu}\|X-Y^k\|_F^2 \\ Y^{k+1} & := & \argmin_{Y \succeq 0} - \tr(F Y) - \langle \Lambda^k, X^{k+1}-Y \rangle + \frac{1}{2\mu}\|X^{k+1}-Y\|_F^2 \\
\Lambda^{k+1} & := & \Lambda^k - (X^{k+1}-Y^{k+1})/\mu, \ea\right.\ee
where $\mu>0$ is a penalty parameter. Following Theorem \ref{the:convergence-admm-generic}, we know that the sequence $\{(X^k,Y^k,\Lambda^k)\}$ generated by \eqref{alg:ADMM-SDR} globally converges to a pair of primal and dual optimal solutions $(X^*,Y^*)$ and $\Lambda^*$ of \eqref{prob:SDR-XY} from any starting point.

It is easy to check that the two subproblems in \eqref{alg:ADMM-SDR} are both relatively easy to solve. Specifically, the solution of the first subproblem in \eqref{alg:ADMM-SDR} corresponds to the projection of $Y^k+\mu\Lambda^k$ onto $\CCal$. The solution of the second problem in \eqref{alg:ADMM-SDR} corresponds to the projection of $X^{k+1}+\mu F - \mu\Lambda^k$ onto the positive semidefinite cone $Y\succeq 0$, i.e.,
\begin{equation*}
Y^{k+1} := U\Diag(\max\{\sigma,0\})U^{\T},
\end{equation*}
where $U\Diag(\sigma)U^{\T}$ is the eigenvalue decomposition of matrix $X^{k+1}+\mu F - \mu\Lambda^k$.

\section{Numerical Results}

\subsection{The ADMM for Convex Programs \eqref{prob:nuclear-penalty} and \eqref{prob:matrix-pca-SDR}}

In this subsection, we report the results on using ADMM \eqref{alg:ADMM-nuclear-norm} to solve the nuclear norm penalty problem \eqref{prob:nuclear-penalty} and ADMM \eqref{alg:ADMM-SDR} to solve the SDP relaxation \eqref{prob:matrix-pca-SDR}. For the nuclear norm penalty problem \eqref{prob:nuclear-penalty}, we choose $\rho=10$.
For ADMM, we choose $\mu=0.5$ and we terminate the algorithms whenever
\[\frac{\|X^{k}-X^{k-1}\|_F}{\|X^{k-1}\|_F} + \|X^{k}-Y^{k}\|_F \leq 10^{-6}.\]
We shall compare ADMM and CVX for solving \eqref{prob:nuclear-penalty} and \eqref{prob:matrix-pca-SDR}, using the default solver of CVX -- SeDuMi version 1.21. We report in Table \ref{tab:cvx-admm-low-dim} the results on randomly created problems with $d=2$ and $n=6,7, 8, 9$. For each pair of $d$ and $n$, we test ten randomly created examples. In Table \ref{tab:cvx-admm-low-dim}, we use `Inst.' to denote the number of the instance and use `Iter.' to denote the number of iterations for ADMM to solve a random instance. We use
`Sol.Dif.' to denote the relative difference of the solutions obtained by ADMM and CVX, i.e., ${\rm Sol.Dif.}=\frac{\|X_{ADMM}-X_{CVX}\|_F}{\max\{1, \|X_{CVX}\|_F\}}$,
and we use `Val.Dif.' to denote the relative difference of the objective values obtained by ADMM and CVX, i.e.,  ${\rm Val.Dif.}=\frac{|v_{ADMM}-v_{CVX}|}{\max\{1, |v_{CVX}|\}}$. We use $T_{ADMM}$ and $T_{CVX}$ to denote the CPU times (in seconds) of ADMM and CVX, respectively. From Table \ref{tab:cvx-admm-low-dim} we see that, ADMM produced comparable solutions compared to CVX; however, ADMM were much faster than CVX, i.e., the interior point solver, especially for nuclear norm penalty problem \eqref{prob:nuclear-penalty}. Note that when $n=10$, ADMM was about 500 times faster than CVX for solving \eqref{prob:nuclear-penalty}, and was about 8 times faster for solving \eqref{prob:matrix-pca-SDR}.

In Table \ref{tab:cvx-admm-high-dim}, we report the results on larger problems, i.e., $n=14, 16, 18, 20$. Because it becomes time consuming to use CVX to solve the nuclear norm penalty problem \eqref{prob:nuclear-penalty} (our numerical tests showed that it took more than three hours to solve one instance of \eqref{prob:nuclear-penalty} for $n=11$ using CVX), we compare the solution quality and objective value of the solution generated by ADMM for solving \eqref{prob:nuclear-penalty} with solution generated by CVX for solving SDP problem \eqref{prob:matrix-pca-SDR}. From Table \ref{tab:cvx-admm-high-dim} we see that, the nuclear norm penalty problem \eqref{prob:nuclear-penalty} and the SDP problem \eqref{prob:matrix-pca-SDR} indeed produce the same solution as they are both close enough to the solution produced by CVX. We also see that using ADMM to solve \eqref{prob:nuclear-penalty} and \eqref{prob:matrix-pca-SDR} were much faster than using CVX to solve \eqref{prob:matrix-pca-SDR}.

\begin{table}[ht]{\footnotesize
\centering
{
\begin{tabular}{|c|c|c|c|c|c|c|c|c|c|c|}
\hline
Inst.  \#&\multicolumn{5}{|c|}{Nuclear Norm Penalty \eqref{prob:nuclear-penalty}} & \multicolumn{5}{|c|}{SDP \eqref{prob:matrix-pca-SDR}}\\\hline
&Sol.Dif.&Val.Dif.&$T_{ADMM}$&Iter. &$T_{CVX}$ & Sol.Dif.&Val.Dif.&$T_{ADMM}$&Iter. &$T_{CVX}$   \\\hline
\multicolumn{11}{|c|}{Dimension $n=6$}\\\hline
 1 & 1.77e-04 & 3.28e-06 & 1.16&  464 & 18.50 & 1.01e-04 & 2.83e-06 & 0.50 &  367 & 1.98\\
 2 & 1.25e-04 & 3.94e-07 & 0.71&  453 & 13.43 & 4.99e-05 & 3.78e-06 & 0.38 &  355 & 1.68\\
 3 & 1.56e-04 & 2.36e-07 & 0.89&  478 & 12.20 & 4.59e-05 & 3.51e-06 & 0.39 &  370 & 1.33\\
 4 & 3.90e-05 & 6.91e-07 & 0.59&  475 & 14.10 & 8.00e-05 & 9.57e-07 & 0.44 &  364 & 2.63\\
 5 & 1.49e-04 & 3.69e-06 & 0.58&  459 & 15.08 & 4.74e-05 & 3.18e-06 & 0.60 &  355 & 1.98\\
 6 & 8.46e-05 & 3.92e-06 & 1.07&  463 & 13.23 & 1.02e-04 & 2.68e-07 & 0.76 &  362 & 1.46\\
 7 & 5.59e-05 & 4.12e-06 & 0.86&  465 & 12.62 & 4.91e-05 & 4.75e-06 & 0.37 &  344 & 1.54\\
 8 & 5.24e-05 & 3.95e-06 & 0.61&  462 & 14.07 & 1.63e-05 & 2.97e-06 & 0.55 &  368 & 1.90\\
 9 & 9.30e-05 & 3.05e-06 & 0.85&  471 & 11.41 & 1.05e-04 & 2.90e-06 & 0.39 &  380 & 1.39\\
10 & 1.36e-04 & 3.89e-08 & 0.56&  465 & 11.04 & 3.38e-05 & 3.11e-06 & 0.30 &  319 & 1.69 \\\hline
\multicolumn{11}{|c|}{Dimension $n=7$}\\\hline
 1 & 1.59e-04 & 4.62e-07 & 1.23&  600 & 65.73 & 1.14e-04 & 4.09e-06 & 0.82 &  453 & 2.60\\
 2 & 9.11e-05 & 3.93e-07 & 1.02&  593 & 68.65 & 8.24e-05 & 2.87e-09 & 0.79 &  474 & 2.51\\
 3 & 2.61e-04 & 4.19e-06 & 1.07&  609 & 66.08 & 6.83e-05 & 4.01e-06 & 0.78 &  480 & 2.53\\
 4 & 1.12e-04 & 4.44e-06 & 1.07&  590 & 65.21 & 6.02e-05 & 3.88e-06 & 0.86 &  480 & 2.50\\
 5 & 1.22e-04 & 4.34e-06 & 1.10&  614 & 57.40 & 9.15e-05 & 4.15e-07 & 0.81 &  487 & 2.57\\
 6 & 1.44e-04 & 8.81e-08 & 1.06&  599 & 60.89 & 4.51e-05 & 4.46e-06 & 0.77 &  466 & 2.44\\
 7 & 1.93e-04 & 3.81e-06 & 1.08&  590 & 66.09 & 1.19e-04 & 2.82e-07 & 0.62 &  389 & 2.54\\
 8 & 1.53e-04 & 4.59e-06 & 1.09&  594 & 59.98 & 2.76e-05 & 3.73e-06 & 0.75 &  463 & 2.61\\
 9 & 1.41e-04 & 4.29e-08 & 1.06&  616 & 78.20 & 3.29e-04 & 4.21e-06 & 0.69 &  443 & 2.57\\
10 & 1.51e-04 & 3.94e-06 & 0.83&  501 & 75.58 & 1.23e-04 & 3.52e-06 & 0.78 &  454 & 2.63 \\\hline
\multicolumn{11}{|c|}{Dimension $n=8$}\\\hline
 1 & 2.86e-04 & 5.10e-06 & 2.15&  728 & 342.25 & 1.12e-04 & 4.52e-06 & 1.59 &  592 & 5.34\\
 2 & 2.76e-04 & 3.95e-07 & 2.07&  739 & 303.75 & 8.17e-05 & 4.78e-06 & 1.81 &  591 & 5.02\\
 3 & 9.29e-05 & 4.78e-06 & 7.74& 2864 & 333.46 & 2.57e-05 & 5.00e-06 & 7.20 & 2746 & 4.75\\
 4 & 3.21e-04 & 4.65e-06 & 2.01&  715 & 337.57 & 9.86e-05 & 4.01e-06 & 1.47 &  512 & 5.00\\
 5 & 1.26e-04 & 7.05e-07 & 1.92&  746 & 335.63 & 7.41e-05 & 4.36e-06 & 1.68 &  607 & 4.92\\
 6 & 1.32e-04 & 1.63e-07 & 2.12&  745 & 336.35 & 7.80e-05 & 5.00e-06 & 1.44 &  550 & 5.29\\
 7 & 3.49e-04 & 7.19e-07 & 2.00&  739 & 309.76 & 6.33e-05 & 4.55e-07 & 1.54 &  582 & 5.03\\
 8 & 4.55e-05 & 4.72e-07 & 2.13&  744 & 316.74 & 3.59e-05 & 7.27e-07 & 1.59 &  600 & 5.02\\
 9 & 5.60e-04 & 4.99e-06 & 2.06&  759 & 336.10 & 4.19e-05 & 4.97e-06 & 1.46 &  569 & 6.00\\
10 & 2.65e-04 & 1.36e-07 & 2.46&  746 & 382.20 & 8.00e-05 & 4.14e-06 & 1.86 &  606 & 5.98\\\hline
\multicolumn{11}{|c|}{Dimension $n=9$}\\\hline
 1 & 1.41e-04 & 1.35e-07 & 4.35&  910 & 1370.60 & 7.29e-05 & 4.78e-06 & 3.26 &  715 & 12.61\\
 2 & 1.83e-04 & 5.77e-06 & 3.60&  872 & 1405.46 & 1.77e-04 & 4.72e-06 & 2.86 &  732 & 9.63\\
 3 & 4.00e-04 & 4.85e-06 & 3.24&  807 & 1709.30 & 3.12e-04 & 8.28e-07 & 2.73 &  702 & 9.99\\
 4 & 3.34e-04 & 1.36e-07 & 3.06&  747 & 1445.57 & 6.13e-05 & 3.19e-07 & 2.91 &  707 & 10.19\\
 5 & 2.63e-04 & 5.43e-06 & 3.62&  904 & 1307.60 & 2.34e-05 & 4.68e-06 & 2.82 &  729 & 10.20\\
 6 & 8.01e-05 & 9.01e-08 & 3.78&  906 & 1353.45 & 9.33e-05 & 5.37e-06 & 2.49 &  597 & 9.31\\
 7 & 2.30e-04 & 5.16e-06 & 3.77&  900 & 1434.71 & 8.14e-05 & 5.68e-06 & 2.75 &  676 & 9.52\\
 8 & 3.27e-04 & 5.45e-06 & 3.71&  908 & 1314.14 & 1.98e-05 & 5.10e-06 & 2.91 &  730 & 9.98\\
 9 & 9.53e-05 & 5.56e-06 & 3.66&  888 & 1575.16 & 1.69e-04 & 4.82e-06 & 2.85 &  714 & 9.64\\
10 & 2.73e-04 & 2.16e-07 & 4.50& 1136 & 1628.80 & 2.73e-05 & 4.98e-06 & 3.39 &  882 & 9.90 \\\hline
\end{tabular}
\caption{Comparison of CVX and ADMM for small-scale problems}
\label{tab:cvx-admm-low-dim}}
}
\end{table}

\begin{table}[ht]{\footnotesize
\centering
{
\begin{tabular}{|c|c|c|c|c|c|c|c|c|c|}
\hline
Inst.  \#&\multicolumn{4}{|c|}{NNP}&\multicolumn{5}{|c|}{SDP}\\
\hline
 &Sol.Dif.$_{DS}$&Val.Dif.$_{DS}$&$T_{ADMM}$ &Iter.& Sol.Dif.&Val.Dif.&$T_{ADMM}$ &Iter. &$T_{CVX}$  \\\hline
\multicolumn{10}{|c|}{Dimension $n=14$}\\\hline
 1 & 4.61e-04 & 8.41e-06 & 36.85 & 1913 & 4.61e-04 & 8.35e-06 & 37.00 & 1621 & 158.21\\
 2 & 4.02e-04 & 2.94e-07 & 39.52 & 1897 & 4.02e-04 & 7.93e-06 & 39.65 & 1639 & 167.89\\
 3 & 1.62e-04 & 2.68e-08 & 37.21 & 1880 & 1.62e-04 & 8.23e-06 & 34.36 & 1408 & 213.04\\
 4 & 4.92e-04 & 7.74e-06 & 45.15 & 1918 & 4.92e-04 & 4.70e-07 & 59.84 & 1662 & 202.95\\
 5 & 8.56e-04 & 8.15e-06 & 34.93 & 1674 & 8.56e-04 & 8.14e-06 & 38.15 & 1588 & 194.01\\
 6 & 3.99e-05 & 4.05e-07 & 34.41 & 1852 & 4.08e-05 & 7.48e-06 & 32.28 & 1411 & 186.99\\
 7 & 7.98e-05 & 7.90e-06 & 38.11 & 1839 & 7.94e-05 & 3.76e-08 & 40.81 & 1555 & 191.76\\
 8 & 1.50e-04 & 8.10e-06 & 38.29 & 1990 & 1.50e-04 & 8.30e-06 & 34.10 & 1543 & 164.13\\
 9 & 1.35e-04 & 8.54e-06 & 34.58 & 1874 & 1.35e-04 & 2.62e-07 & 30.33 & 1387 & 171.77\\
10 & 5.50e-04 & 8.59e-06 & 37.28 & 1825 & 5.50e-04 & 7.71e-06 & 35.85 & 1567 & 169.51\\\hline
\multicolumn{10}{|c|}{Dimension $n=16$}\\\hline
 1 & 5.22e-05 & 9.00e-06 & 125.24 & 2359 & 5.21e-05 & 9.45e-06 & 102.85 & 2035 & 582.19\\
 2 & 1.02e-04 & 3.37e-07 & 92.37 & 2244 & 1.02e-04 & 9.11e-06 & 63.02 & 1427 & 606.70\\
 3 & 2.02e-05 & 5.97e-07 & 96.21 & 2474 & 2.01e-05 & 4.40e-07 & 83.92 & 1910 & 566.92\\
 4 & 8.53e-05 & 9.27e-06 & 90.83 & 2323 & 8.54e-05 & 9.59e-06 & 93.44 & 2048 & 560.54\\
 5 & 2.14e-04 & 9.19e-06 & 86.22 & 2359 & 2.14e-04 & 2.19e-07 & 80.06 & 1961 & 523.15\\
 6 & 3.12e-04 & 9.29e-06 & 88.82 & 2304 & 3.12e-04 & 8.58e-06 & 88.31 & 2042 & 498.55\\
 7 & 9.69e-05 & 9.12e-06 & 88.29 & 2431 & 9.65e-05 & 2.86e-07 & 88.05 & 2067 & 520.82\\
 8 & 3.34e-04 & 1.00e-05 & 85.32 & 2271 & 3.34e-04 & 8.53e-06 & 85.04 & 2043 & 515.85\\
 9 & 2.61e-04 & 9.01e-06 & 93.13 & 2475 & 2.61e-04 & 9.12e-06 & 88.85 & 2034 & 505.71\\
10 & 2.06e-04 & 3.45e-07 & 103.92 & 2813 & 2.05e-04 & 1.01e-05 & 94.41 & 2269 & 527.50  \\\hline
\multicolumn{10}{|c|}{Dimension $n=18$}\\\hline
 1 & 2.70e-04 & 1.01e-05 & 172.97 & 2733 & 2.70e-04 & 1.87e-07 & 168.91 & 2323 & 1737.94\\
 2 & 8.17e-04 & 1.11e-05 & 184.70 & 2970 & 8.17e-04 & 1.99e-07 & 168.83 & 2365 & 1549.10\\
 3 & 1.07e-04 & 3.22e-08 & 183.72 & 2920 & 1.07e-04 & 1.14e-05 & 169.64 & 2456 & 1640.04\\
 4 & 5.16e-04 & 1.01e-05 & 182.40 & 2958 & 5.16e-04 & 1.02e-05 & 174.72 & 2442 & 1636.86\\
 5 & 9.48e-04 & 1.03e-05 & 184.69 & 3039 & 9.48e-04 & 1.04e-05 & 170.68 & 2441 & 1543.41\\
 6 & 1.67e-04 & 1.03e-05 & 171.71 & 2845 & 1.67e-04 & 9.96e-06 & 182.37 & 2553 & 1633.55\\
 7 & 4.87e-05 & 3.77e-07 & 180.64 & 2883 & 4.87e-05 & 2.79e-07 & 187.56 & 2545 & 1638.38\\
 8 & 8.28e-05 & 1.07e-05 & 178.35 & 2904 & 8.28e-05 & 1.04e-05 & 181.57 & 2542 & 1641.56\\
 9 & 2.45e-04 & 1.06e-07 & 174.82 & 2902 & 2.45e-04 & 9.97e-06 & 152.58 & 2127 & 1735.26\\
10 & 9.58e-05 & 7.61e-07 & 191.06 & 2872 & 9.66e-05 & 1.11e-05 & 183.29 & 2480 & 1642.33  \\\hline
\multicolumn{10}{|c|}{Dimension $n=20$}\\\hline
 1 & 1.23e-03 & 6.98e-08 & 414.62 & 3415 &1.23e-03 & 4.21e-08 & 388.36 & 2810 &6116.02\\
 2 & 7.93e-04 & 1.24e-05 & 401.54 & 3383 &7.93e-04 & 1.14e-05 & 347.27 & 2689 &6182.56\\
 3 & 3.11e-04 & 1.21e-05 & 426.91 & 3498 &3.11e-04 & 1.21e-05 & 399.92 & 2845&6808.99\\
 4 & 7.16e-05 & 6.99e-07 & 397.69 & 3312 &7.40e-05 & 1.18e-05 & 366.82 & 2758 &7701.91\\
 5 & 6.24e-04 & 1.19e-05 & 435.05 & 3564 & 6.25e-04 & 1.20e-05 & 419.23 & 2903 & 7419.43\\
 6& 1.09e-04 & 1.20e-05 & 393.25 & 3376 & 1.09e-04 & 1.15e-05 & 397.43 & 2869 & 8622.19 \\
7 & 4.58e-04 & 3.21e-05 & 429.38 & 3536 & 4.58e-04 & 3.20e-05 & 422.72 & 2938 & 9211.37\\
8 & 6.15e-04 & 1.11e-05 & 273.33 & 2330 & 6.15e-04 & 7.14e-07 & 205.49 & 1511 & 5166.66\\
9 & 4.92e-04 & 1.16e-05 & 344.99 & 3017 & 4.92e-04 & 2.32e-07 & 259.18 & 1896 & 5063.00 \\
10 & 3.45e-004 & 2.56e-004 & 395.63 & 3357 & 1.14e-005  &  4.36e-007&  359.13 & 2713 & 6559.39 \\\hline
\end{tabular}
\caption{Comparison of CVX and ADMM for large-scale problems}
\label{tab:cvx-admm-high-dim}}
}
\end{table}

\subsection{Comparison with SOS { and MBI}}

Based on the results of the above tests, we may conclude that it is most efficient to solve the SDP relaxation by ADMM. In this subsection, we compare this approach with 
{two competing methods: one is based on the Sum of Squares (SOS) approach (Lasserre \cite{L01,L01_2} and Parrilo \cite{P00,P03}), and the other one is the Maximum Block Improvement (MBI) method proposed by Chen {\it et al.}~\cite{CHLZ11}.

Theoretically speaking, the SOS can solve any
general polynomial problems to any given accuracy, but it }requires to solve a sequence of (possibly large) semidefinite programs, which limits the applicability of the method to solve large size problems.
Henrion et al.~\cite{HLL09}
developed a specialized Matlab toolbox known as GloptiPoly~3 based on SOS approach, which will be used in our tests. {The MBI is tailored for multi-block optimization problem, and the polynomial optimization can be treated as multi-block problems, to which MBI can be applied. As we mentioned before, MBI aims to finding a stationary point, which may or may not be globally optimal.}

In Table~\ref{tab:SDP-MBI-GLP} we report the results using ADMM to solve SDP relaxation of PCA problem and compare them with the results of applying the SOS method {as well as the MBI method} for the same problem. {When using the MBI, as suggested in~\cite{CHLZ11}, we actually work on an equivalent problem of~\eqref{prob:tensor-pca}:  $\max\limits_{\| x \| = 1} \mathcal{F}(\underbrace{x,\cdots,x}_{2d})+ 6(x^{\top}x)^{d}$, where the equivalence is due to the constraint $\| x \| = 1$. This transformation can help the MBI avoid getting trapped in a local minimum.}

We use `Val.' to denote the objective value of the solution, `Status' to denote optimal status of GloptiPoly $3$, i.e., ${\rm Status}=1$ means GloptiPoly~$3$ successfully identified the optimality of current solution, `Sol.R.' to denote the solution rank returned by SDP relaxation and thanks to the previous discussion `Sol.R.=1' means the current solution is already optimal. From Table~\ref{tab:SDP-MBI-GLP}, {we see that the MBI is the fastest among all the methods but usually cannot guarantee global optimality, while GloptiPoly~$3$ is very time consuming but can globally solve most instances.} Note that when $n=20$, our ADMM was about 30 times faster than GloptiPoly~$3$. Moreover, for some instances GloptiPoly $3$ cannot identify the optimality even though the current solution is actually already optimal (see {instance $9$ with $n=16$ and instance $3$ with $n=18$}).

\begin{table}[ht]{\footnotesize
\centering
\begin{tabular}{|c|c|c|c|c|c|c|c|c|}\hline
Inst.  \#&\multicolumn{2}{|c|}{MBI} &\multicolumn{3}{|c|}{GLP}&\multicolumn{3}{|c|}{SDP by ADMM}\\\hline
& {Val.} & {Time} & Val. & Time &Status & Val.& Time & Sol.R. \\\hline
\multicolumn{9}{|c|}{Dimension $n=14$}\\\hline
 1 & { 5.17 } & { 0.23} & 5.28 & 143.14 & 1 & 5.28 & 14.29 & 1\\
 2 & { 5.04 } & { 0.22} & 5.65 & 109.65 & 1 & 5.65 & 32.64 & 1\\
 3 & { 5.08 } & { 0.13} & 5.80 & 119.48 & 1 & 5.80 & 34.30 & 1\\
 4 & { 5.94 } & { 0.16} & 5.95 & 100.39 & 1 & 5.95 & 30.64 & 1\\
 5 & { 4.74 } & { 0.48} & 5.88 & 122.19 & 1 & 5.88 & 33.13 & 1\\
 6 & { 5.68 } & { 0.54} & 6.38 & 122.44 & 1 & 6.38 & 33.30 & 1\\
 7 & { 4.61 } & { 0.12} & 5.91 & 104.68 & 1 & 5.91 & 30.17 & 1\\
 8 & { 5.68 } & { 0.23} & 6.31 & 141.52 & 1 & 6.31 & 41.73 & 1\\
 9 & { 5.93 } & { 0.22} & 6.40 & 102.73 & 1 & 6.40 & 37.32 & 1\\
10 & { 5.09 } & { 0.36} & 6.03 & 114.35 & 1 & 6.03 & 35.68 & 1 \\\hline
\multicolumn{9}{|c|}{Dimension $n=16$}\\\hline
 1 & { 6.52 } & { 0.45} & 6.74 & 420.10 & 1 & 6.74 & 91.80 & 1\\
 2 & { 5.51 } & { 1.21} & 5.93 & 428.10 & 1 & 5.93 & 83.90 & 1\\
 3 & { 5.02 } & { 0.30} & 6.44 & 393.16 & 1 & 6.44 & 90.16 & 1\\
 4 & { 5.60 } & { 0.32} & 6.48 & 424.07 & 1 & 6.48 & 90.67 & 1\\
 5 & { 5.78 } & { 0.36} & 6.53 & 431.44 & 1 & 6.53 & 95.48 & 1\\
 6 & { 5.23 } & { 0.26} & 6.42 & 437.58 & 1 & 6.42 & 98.19 & 1\\
 7 & { 6.11 } & { 0.24} & 6.23 & 406.16 & 1 & 6.23 & 89.21 & 1\\
 8 & { 5.92 } & { 0.51} & 6.39 & 416.58 & 1 & 6.39 & 89.75 & 1\\
 9 & { 5.47 } & { 0.28} & 6.00 & 457.29 & 0 & 6.00 & 77.56 & 1\\
10 & { 4.95 } & { 0.35} & 6.32 & 367.26 & 1 & 6.32 & 80.38 & 1 \\\hline
\multicolumn{9}{|c|}{Dimension $n=18$}\\\hline
 1 & { 6.16 } & { 0.57} & 7.38 & 1558.00 & 1 & 7.38 & 199.44 & 1\\
 2 & { 5.94 } & { 0.25} & 6.65 & 1388.45 & 1 & 6.65 & 190.52 & 1\\
 3 & { 7.42 } & { 0.22} & 7.42 & 1500.05 & 0 & 7.42 & 193.27 & 1\\
 4 & { 5.85 } & { 0.94} & 7.21 & 1481.34 & 1 & 7.21 & 195.02 & 1\\
 5 & { 7.35 } & { 0.43} & 7.35 & 1596.00 & 1 & 7.35 & 117.44 & 1\\
 6 & { 5.91 } & { 1.05} & 6.79 & 1300.82 & 1 & 6.78 & 193.36 & 1\\
 7 & { 5.80 } & { 0.85} & 6.84 & 1433.50 & 1 & 6.84 & 182.58 & 1\\
 8 & { 5.72 } & { 0.54} & 6.96 & 1648.63 & 1 & 6.96 & 231.88 & 1\\
 9 & { 6.15 } & { 0.17} & 7.07 & 1453.82 & 1 & 7.07 & 212.50 & 1\\
10 & { 6.01 } & { 1.11} & 6.89 & 1432.06 & 1 & 6.89 & 199.26 & 1 \\\hline
\multicolumn{9}{|c|}{Dimension $n=20$}\\\hline
 1 & { 5.95 } & { 0.39} & 7.40 & 8981.97 & 1 & 7.40 & 429.64 & 1\\
 2 & { 6.13 } & { 2.14} & 6.93 & 9339.06 & 1 & 6.93 & 355.25 & 1 \\
 3 & { 6.37 } & { 2.49} & 6.68 & 9629.04 & 1 & 6.68 & 418.11 & 1\\
 4 & { 6.23 } & { 1.14} & 6.87 & 10148.21 & 1 & 6.87 & 404.18 & 1 \\
  5 & { 6.62 } & { 1.66} & 7.72 & 11079.94 & 1 & 7.72 & 326.44 & 1\\
 6 & { 6.81 } & { 1.26} & 7.46 & 10609.65 & 1 & 7.46 & 415.69 & 1 \\
 7 & { 7.80 } & { 1.02} & 7.80 & 9723.37 & 1 & 7.80 & 430.76 & 1 \\
8 & { 6.03 } & { 0.95} & 7.02 & 12755.35 & 1 & 7.02 & 416.00 & 1 \\
9& { 7.80 } & { 0.61} & 7.80 &  12353.47 & 1 & 7.80 & 430.45 & 1\\
 10 & { 7.47 } & { 0.89} & 7.47 &  11629.12 & 1 & 7.47 & 375.52 & 1 \\ \hline

\end{tabular}
\caption{{Comparison SDP Relaxation by ADMM with GloptiPoly~3 and MBI.}}
\label{tab:SDP-MBI-GLP}}
\end{table}
{
\subsection{Comparison with Z-Eigenvalue Methods}
In~\cite{QWW09}, Qi {\it et al.} proposed two heuristic methods to find the maximum Z-eigenvalue of the third order super-symmetric tensors.
We will show later in Section~\ref{tri-linear} that our method can solve tri-linear (not necessary super-symmetric) tensor PCA problems. Thus in this subsection, we report the results of using ADMM to solve SDP relaxation of the third order tensor PCA problems and compare them with the results of applying
the two Z-eigenvalue methods, which are referred as ``Z1'' and ``Z2'' in Tables~\ref{tab:rank1-frequency-eigenvalue-method-normal} and~\ref{tab:rank1-frequency-eigenvalue-method-uniform}.
In Table~\ref{tab:rank1-frequency-eigenvalue-method-normal} we generate 1000 tensors by normal distribution, while in~Table~\ref{tab:rank1-frequency-eigenvalue-method-uniform}, the 1000 instances are generated by uniform distribution in the interval $(-1,1)$. We report, in Tables~\ref{tab:rank1-frequency-eigenvalue-method-normal} and~\ref{tab:rank1-frequency-eigenvalue-method-uniform}, the number of instances that are globally solved by each algorithm. According to these experiments, we can see that the performance of our approach is better than that of Z1 and is comparable to Z2.
}

\begin{table}[htb]\small
\centering
{
\begin{tabular}{c|r|r|r}\hline
$n$ & SDP & Z1& Z2 \\\hline
4& 999 & 979 & 995\\
5& 999 & 964 & 993\\
 6&  1000 & 947 & 998\\
 7&  1000 & 941 & 997\\
8 & 1000 & 938 & 999\\
9 & 1000 & 911 & 997 \\
 10 & 1000 & 906 & 1000 \\
 \hline
\end{tabular}
\caption{Comparison with Z-eigenvalue methods
with data generated by normal distribution}
\label{tab:rank1-frequency-eigenvalue-method-normal}
}
\end{table}

\begin{table}[htb]\small
\centering
{
\begin{tabular}{c|r|r|r}\hline
$n$& SDP & Z1& Z2 \\\hline
4 &  999 & 986 & 995 \\
 5 &  1000 & 966 & 997 \\
  6 &  999 & 945 & 997\\
  7 & 999 & 934 & 1000 \\
  8 &  1000 & 902 & 999 \\
   9 &  1000 & 910 & 999 \\
   10 & 1000 & 898 & 1000 \\
 \hline
\end{tabular}
\caption{Comparison with Z-eigenvalue methods with data generated by uninform distribution}
\label{tab:rank1-frequency-eigenvalue-method-uniform}
}
\end{table}

{
\section{What If the Solution Is Not Rank-One?}

Although our numerical results strongly indicate that problems \eqref{prob:nuclear-penalty} and \eqref{prob:matrix-pca-SDR} {are very likely to} admit rank-one solutions (100\% for the randomly created problems we tested), it is in principle possible that the solution $X^*= \sum\limits_{i=1}^{r}a^i(a^i)^{\top}$ is not of rank one, i.e., $r > 1$. In this situation, we can introduce a small perturbation to the original tensor $\mathcal{F}$ and apply the proposed algorithms. If the newly obtained solution is rank-one, we can say that this solution is a good approximation of the ``true'' solution. Another way to proceed is to apply a post-processing procedure, which will be discussed below, to $X^*$ to obtain a rank-one solution.

We denote $\mathcal{X}^*=\matr^{-1}(X^*)$, and $\mathcal{X}^*=\sum\limits_{i=1}^{r}\mathcal{A}^i\otimes \mathcal{A}^i $, where $\vect(\mathcal{A}^i)=a^{i}$. Basically, we
want to find the projection of $\mathcal{X}^*$ onto the rank-one tensor set $\{ \mathcal{X} \in \SI^{2d} \; \big{|}\; \rank(\mathcal{X})=1, \|\mathcal{X} \|_F=1 \}$:
\[\min\limits_{\mathcal{X} \in \SI^{2d}, \|\mathcal{X}\|_F=1, \rank(\mathcal{X})=1}\| \mathcal{X}^*-\mathcal{X}\|_{F}, \]
which is equivalent to
\begin{equation}\label{Prob:rank-one-projection}
\max\limits_{\|x^1\|=\cdots=\|x^{2d}\|=1}\mathcal{X}^*\left(x^1,x^2,\cdots,x^{2d}\right).
\end{equation}
This is a problem in the form of~\eqref{prob:tensor-pca}, but the difference is that $\mathcal{X}^*$ has a further structure which plays an important role in the later discussion.
\begin{proposition}\label{tensor-PSD}
For a tensor $\mathcal{F} =\sum\limits_{i=1}^{r}\mathcal{A}^i\otimes \mathcal{A}^i \in \SI^{n^{2d}}$, it holds that
\begin{equation}\label{co-quadratic-PSD}
\mathcal{F}\left(x^1,x^1,x^2,x^2,\cdots,x^d,x^d\right) \ge 0, \quad \forall\;x^1,x^2,\cdots,x^d.
\end{equation}
\end{proposition}
\begin{proof} Since $\cal F$ is super-symmetric, for any $x^1,x^2,\cdots,x^d$,
\begin{eqnarray*}
\mathcal{F}\left(x^1,x^1,x^2,x^2,\cdots,x^d,x^d\right)&=&\mathcal{F}\left(x^1,x^2,\cdots,x^d,x^1,x^2,\cdots,x^d\right)\\
&=&\sum\limits_{i=1}^{r}\mathcal{A}^i\otimes \mathcal{A}^i\left(x^1,x^2,\cdots,x^d,x^1,x^2,\cdots,x^d\right)\\
&=&\sum\limits_{i=1}^{r}\left(\mathcal{A}^i\left(x^1,x^2,\cdots,x^d\right)\right)^2\ge0.
\end{eqnarray*}
\end{proof}
Inequality~\eqref{co-quadratic-PSD} is called co-quadratic positive semidefinite. In~\cite{C12}, it was proved that if $\cal F$ is co-quadratic positive semidefinite then
\begin{equation*}
\mathcal{F}\left(x^1,x^2,\cdots,x^{2d}\right) \le \max\limits_{1 \le i \le 2d}\left\{\mathcal{F}(\underbrace{x^i,\cdots,x^i}_{2d})\right\}.
\end{equation*}
As a result
\begin{equation}\label{multilinear-equivalence}
\max\limits_{\|x\|=1}\mathcal{F}(\underbrace{x,\cdots,x}_{2d})= \max\limits_{\|x^1\|=\cdots=\|x^{2d}\|=1}\mathcal{F}\left(x^1,x^2,\cdots,x^{2d}\right).
\end{equation}
For a solution $X^*$, which is either optimal to \eqref{prob:nuclear-penalty} or \eqref{prob:matrix-pca-SDR}, by Proposition~\ref{tensor-PSD} we know that $\mathcal{X}^* =\matr^{-1}(X^*)$ is co-quadratic positive semidefinite. So the problem~\eqref{Prob:rank-one-projection} on finding the rank-one projection of $\mathcal{X}^*$ is equivalent to
\[\max\limits_{\|x^1\|=\cdots=\|x^{2d}\|=1}\mathcal{X}^*\left(x^1,x^2,\cdots,x^{2d}\right)\]
due to equality~\eqref{multilinear-equivalence}. Essentially, we can resort to a multi-linear problem, which is easier than~\eqref{prob:tensor-pca} and can be solved, for instance, by the MBI method proposed in~\cite{CHLZ11}.
Note that if we apply the MBI method directly to ~\eqref{prob:tensor-pca}, relation~\eqref{multilinear-equivalence} may not be guaranteed: the MBI may get trapped in a local minimum instead of local maximum.
}

\section{Extensions}

In this section, we show how to extend the results in the previous sections for super-symmetric tensor PCA problem to tensors that are not super-symmetric.

\subsection{Bi-quadratic tensor PCA}\label{bi-quadratic-PCA}

A closely related problem to the tensor PCA problem~\eqref{prob:tensor-pca} is the following bi-quadratic PCA problem:
\begin{equation}\label{prob:biquadratic}
\begin{array}{ll} \max & \mathcal{G}(x,y,x,y) \\ \st & x\in \RR^{n},\|x\|=1,\\
                                                   & y\in \RR^{m},\|y\|=1,
\end{array}
\end{equation}
where $\mathcal G $ is a {\it partial-symmetric}\/ tensor defined as follows.
\begin{definition}
A $4$-th order tensor $\mathcal G \in \RR^{{(nm)^2}}$ is called partial-symmetric if $\mathcal{G}_{ijk\ell}=\mathcal{G}_{kji\ell}=\mathcal{G}_{i\ell kj}, \forall i,j,k,\ell$. The space of all $4$-th order partial-symmetric tensor is denoted by $\overrightarrow{\overrightarrow{\SI}}^{(mn)^2}$.
\end{definition}
Various approximation algorithms for bi-quadratic PCA problem have been studied in~\cite{LNWY09}. Problem~\eqref{prob:biquadratic} arises from the strong ellipticity condition problem in solid mechanics and the entanglement problem in quantum physics; see~\cite{LNWY09} for more applications of bi-quadratic PCA problem.

{We can unfold a partial-symmetric tensor $\mathcal{G}$ in a similar manner as in Definition \ref{def:matricization}.
\begin{definition}\label{def:matricization-partial-symmetric}
For $\mathcal{G} \in \overrightarrow{\overrightarrow{\SI}}^{{(nm)^2}}$, we define its square matrix rearrangement, denoted by $\matr(\mathcal{G})\in \RR^{mn\times mn}$, as the following:
\[\matr(\mathcal{G})_{k \ell}:=\mathcal{G}_{i_1i_{2}i_{3}i_{4}}, \; 1\leq i_1,i_{3} \leq n,\;1\leq i_2,i_{4} \leq m\;\mbox{where}\;k =(i_1-1)m+i_2,\mbox{ and }\ell = (i_3-1)m+i_4.\]
\end{definition}
}

It is easy to check that for given vectors $a \in \RR^n$ and $b \in \RR^m$, $a \otimes b \otimes a \otimes b \in \overrightarrow{\overrightarrow{\SI}}^{{(nm)^2}}$. Moreover, we say a partial-symmetric tensor $\mathcal{G}$ is of rank one if $\mathcal{G}={\lambda}\,a \otimes b \otimes a \otimes b$ for some vectors {$a,b$ and scaler $\lambda$}.

Since $\tr(xy^{\T}yx^{\T})=x^{\T}xy^{\T}y=1$, by letting $\mathcal{X}=x \otimes y \otimes x \otimes y$, problem \eqref{prob:biquadratic} is equivalent to
\begin{equation*}
\begin{array}{ll} \max & \mathcal{G} \bullet \mathcal{X} \\ \st & \sum\limits_{i, j}\mathcal{X}_{ijij} =1,\\
                                & \mathcal{X} \in \overrightarrow{\overrightarrow{\SI}}^{{(nm)^2}},\;\rank(\mathcal{X})=1.
\end{array}
\end{equation*}

In the following, we group variables $x$ and $y$ together and treat $x \otimes y$ as a long vector by stacking its columns.
Denote $X=\matr(\mathcal{X})$ and $G = \matr(\mathcal{G})$. Then, we end up with a matrix version of the above tensor problem:
\begin{equation}\label{prob:matrix-biquadratic}
\begin{array}{ll} \max & \tr(G X) \\ \st & \tr(X) =1,\quad X \succeq 0,\\
                                & \matr^{-1}(X) \in \overrightarrow{\overrightarrow{\SI}}^{{(nm)^2}},\;\rank(X)=1.
\end{array}
\end{equation}

As it turns out, the rank-one equivalence theorem can be extended to the partial symmetric tensors. Therefore the above two problems are actually equivalent.
\begin{theorem}\label{theorem:partical-symmetric-rank1} Suppose $A$ is an $n\times m$ dimensional matrix. Then the following two statements are equivalent:\\
(i) $\rank(A)=1$;\\
(ii) $A \otimes A \ \in \overrightarrow{\overrightarrow{\SI}}^{{(nm)^2}}$.\\
In other words, suppose $\mathcal{F} \in  \overrightarrow{\overrightarrow{\SI}}^{{(nm)^2}}$, then $\rank (\mathcal{F})=1 \Longleftrightarrow \rank (F)=1$, where $F = \matr(\mathcal{F})$.
\end{theorem}
\begin{proof} (i) $\Longrightarrow$ (ii) is obvious. Suppose $\rank(A)=1$, say $A=ab^{\T}$ for some $a\in \RR^n$ and $b \in \RR^m$. Then $\mathcal{G}=A \otimes A = a \otimes b \otimes a \otimes b$ is partial-symmetric.

Conversely, suppose $\mathcal{G}=A \otimes A \in  \overrightarrow{\overrightarrow{\SI}}^{{(nm)^2}}$. Then
$$
A_{i_1j_1}A_{i_2j_2}=\mathcal{G}_{i_1j_1i_2j_2}=\mathcal{G}_{i_2j_1i_1j_2}=A_{i_2j_1}A_{i_1j_2}, \;\forall 1\le i_1,i_2 \le n,\;1\le j_1,j_2 \le m,
$$
implies $A_{i_1j_1}A_{i_2j_2} - A_{i_2j_1}A_{i_1j_2}=0$. That is, every $2 \times 2$ minor of matrix $A$ is zero. Thus $A$ is of rank one.
\end{proof}

By using the similar argument in Theorem~\ref{thoerem:SDR-NuclearNormPenalty}, we can show that the following SDP relaxation of~\eqref{prob:matrix-biquadratic} has a good chance to get a low rank solution.
\begin{equation}\label{prob:relax-matrix-biquadratic}
\begin{array}{ll} \max & \tr(G X) \\ \st & \tr(X) =1,\quad X \succeq 0,\\
                                & \matr^{-1}(X) \in \overrightarrow{\overrightarrow{\SI}}^{{(nm)^{2}}}.
\end{array}
\end{equation}
Therefore, we used the same ADMM to solve~\eqref{prob:relax-matrix-biquadratic}. The frequency of returning rank-one solutions for randomly created examples is reported in Table~\ref{tab:rank1-frequency-biquadratic}. As in Table~\ref{tab:rank1-frequency-NNP} and Table~\ref{tab:rank1-frequency-SDR}, we tested $100$ random instances for each $(n,m)$ and report the number of instances that produced rank-one solutions. We also report the average CPU time (in seconds) using ADMM to solve the problems. Table~\ref{tab:rank1-frequency-biquadratic} shows that the SDP relaxation \eqref{prob:relax-matrix-biquadratic} can give a rank-one solution for {\it most} randomly created instances, and thus { is likely to} solve the original problem~\eqref{prob:biquadratic} to optimality.
 \begin{table*}[htb]\small
\centering
\begin{tabular}{c|c|c}
\hline
Dim ($n,m$) &$\rank$-1& CPU\\
\hline
 (4,4) & 100 & 0.12\\
 (4,6) & 100 & 0.25 \\
 (6,6) & 100 & 0.76 \\
 (6,8) & 100 & 1.35 \\
 (8,8) & 98 & 2.30 \\
 (8,10) & 100 & 3.60\\
 (10,10) & 96 & 5.77\\
\hline
\end{tabular}
\caption{Frequency of problem~\eqref{prob:relax-matrix-biquadratic} having rank-one solution}
\label{tab:rank1-frequency-biquadratic}
\end{table*}

\subsection{Tri-linear tensor PCA}\label{tri-linear}
Now let us consider a highly non-symmetric case: tri-linear PCA.
\begin{equation}\label{prob:tri-linear}
\begin{array}{ll} \max & \mathcal{F}(x,y,z) \\ \st & x\in \RR^{n},\|x\|=1,\\
                                                   & y\in \RR^{m},\|y\|=1,\\
                                                   & z\in \RR^{\ell},\;\, \|z\|=1,
\end{array}
\end{equation}
where $\mathcal {F} \in \RR^{n \times m \times \ell}$ is any $3$-rd order tensor and $n \le m \le \ell$.

Recently, tri-linear PCA problem was found to be very useful in many practical problems.
For instance, Wang and Ahuja~\cite{WA04} proposed a tensor rank-one decomposition algorithm to compress image sequence, where they essentially solve a sequence of tri-linear PCA problems.

By the Cauchy-Schwarz inequality, the problem~\eqref{prob:tri-linear} is equivalent to
\begin{equation*}
\begin{array}{ll} \max & \|\mathcal{F}(x,y,\cdot)\| \\ \st & x\in \RR^{n},\|x\|=1,\\
                                                   & y\in \RR^{m},\|y\|=1,
\end{array}
\Longleftrightarrow
\begin{array}{ll} \max & \|\mathcal{F}(x,y,\cdot)\|^2 \\ \st & x\in \RR^{n},\|x\|=1,\\
                                                   & y\in \RR^{m},\|y\|=1.
\end{array}
\end{equation*}
We further notice
\begin{eqnarray*}
\|\mathcal{F}(x,y,\cdot)\|^2
&=& \mathcal{F}(x,y,\cdot)^{\T}\mathcal{F}(x,y,\cdot)=\sum\limits_{k=1}^{\ell}\mathcal{F}_{ijk}\,\mathcal{F}_{uvk}\,x_iy_jx_uy_v \\
&=&\sum\limits_{k=1}^{\ell}\mathcal{F}_{ivk}\,\mathcal{F}_{ujk}\,x_iy_vx_uy_j=\sum\limits_{k=1}^{\ell}\mathcal{F}_{ujk}\,\mathcal{F}_{ivk}\,x_uy_jx_iy_v.
\end{eqnarray*}
Therefore, we can find a partial symmetric tensor $\mathcal{G}$ with
$$\mathcal{G}_{ijuv}=\sum\limits_{k=1}^{\ell}\left(\mathcal{F}_{ijk}\,\mathcal{F}_{uvk} + \mathcal{F}_{ivk}\,\mathcal{F}_{ujk}+\mathcal{F}_{ujk}\,\mathcal{F}_{ivk}\right)/3,\,\,\,\,\forall \; i,j,u,v,$$
such that $\|\mathcal{F}(x,y,\cdot)\|^2  = \mathcal{G}\left( x,y,x,y \right)$. Hence, the tri-linear problem~\eqref{prob:tri-linear} can be equivalently formulated in the form of problem~\eqref{prob:biquadratic}, which can be solved by the method proposed in the previous subsection.

\subsection{Quadri-linear tensor PCA}

In this subsection, we consider the following quadri-linear PCA problem:
\begin{equation}\label{prob:quadri-linear-pca}
\begin{array}{ll} \max & \mathcal{F}(x^1,x^2,x^3,x^4) \\ \st & x^i \in \RR^{n_i}, \|x^i\|=1,\;\forall\; i=1,2,3,4,
\end{array}
\end{equation}
where $\mathcal{F} \in \RR^{n_1 \times \cdots \times n_{4}}$ with $n_1 \le n_3 \le n_2 \le n_4$.
Let us first convert the quadri-linear function $\mathcal{F}(x^1,x^2,x^3,x^4)$ to a bi-quadratic function $\mathcal{T}\left({{x^1} \atop{x^3}},{{x^2} \atop{x^4}},{{x^1} \atop{x^3}},{{x^2} \atop{x^4}}\right)$ with $\mathcal{T}$ being partial symmetric. To this end, we first construct $\mathcal {G}$ with
$$
\mathcal{G}_{i_1,i_2, n+i_3, n+i_4}=\left\{
\begin{array}{cl} \mathcal{F}_{j_1j_2 j_3 j_4}, & \mbox{if}\;1 \le i_k \le n_k \\
0, & \mbox{otherwise}.
 \end{array}\right.
$$
Consequently, we have $\mathcal{F}(x^1,x^2,x^3,x^4) = \mathcal{G}\left({{x^1} \atop{x^3}},{{x^2} \atop{x^4}},{{x^1} \atop{x^3}},{{x^2} \atop{x^4}}\right)$. Then we can further partial-symmetrize $\mathcal{G}$ and the desired tensor $\mathcal{T}$ is as follows,
\[
\mathcal{T}_{i_1i_2 i_3 i_4}= \frac{1}{4}\left(\mathcal{G}_{i_1i_2 i_3 i_4}+\mathcal{G}_{i_1i_4 i_3 i_2}+\mathcal{G}_{i_3i_2 i_1 i_4}+\mathcal{G}_{i_3i_4 i_1 i_2}  \right)\quad\forall\; i_1,i_{2},i_3,i_4,
\]
satisfying $\mathcal{T}\left({{x^1} \atop{x^3}},{{x^2} \atop{x^4}},{{x^1} \atop{x^3}},{{x^2} \atop{x^4}}\right)=\mathcal{G}\left({{x^1} \atop{x^3}},{{x^2} \atop{x^4}},{{x^1} \atop{x^3}},{{x^2} \atop{x^4}}\right)$.
Therefore, problem~\eqref{prob:quadri-linear-pca} is now reformulated as a bi-quadratic problem:
\begin{equation}\label{prob:biquadratic-sphere-a}
\begin{array}{ll} \max & \mathcal{T}\left({{x^1} \atop{x^3}},{{x^2} \atop{x^4}},{{x^1} \atop{x^3}},{{x^2} \atop{x^4}}\right) \\ \st & x^i \in \RR^{n_i}, \|x^i\|=1,\;\forall\; i=1,\ldots,4.
\end{array}
\end{equation}

Moreover, we can show that the above problem is actually a bi-quadratic problem in the form of~\eqref{prob:biquadratic}.
\begin{proposition}Suppose $\mathcal{T}$ is a fourth order partial symmetric tensor. Then
problem~\eqref{prob:biquadratic-sphere-a} is equivalent to
\begin{equation}\label{prob:biquadratic-sphere-b}
\begin{array}{ll} \max & \mathcal{T}\left({{x^1} \atop{x^3}},{{x^2} \atop{x^4}},{{x^1} \atop{x^3}},{{x^2} \atop{x^4}}\right) \\ \st &\sqrt{\|x^1\|^2+ \|x^3\|^2}=\sqrt{2},\\&\sqrt{\|x^2\|^2+ \|x^4\|^2}=\sqrt{2}.
\end{array}
\end{equation}
\end{proposition}
\begin{proof} It is obvious that~\eqref{prob:biquadratic-sphere-b} is a relaxation of~\eqref{prob:biquadratic-sphere-a}. To further prove that the relaxation~\eqref{prob:biquadratic-sphere-b} is tight, we assume $(\hat{x}^1, \hat{x}^2, \hat{x}^3, \hat{x}^4)$ is optimal to~\eqref{prob:biquadratic-sphere-b}. Then $\mathcal{T}\left({{\hat{x}^1} \atop{\hat{x}^3}},{{\hat{x}^2} \atop{\hat{x}^4}},{{\hat{x}^1} \atop{\hat{x}^3}},{{\hat{x}^2} \atop{\hat{x}^4}}\right)= \mathcal{F}(\hat{x}^1,\hat{x}^2,\hat{x}^3,\hat{x}^4)>0$, and so $\hat{x}^i \neq 0$ for all $i$. Moreover, notice that
$$
\sqrt{\|\hat{x}^1\| \|\hat{x}^3\|} \le \sqrt{ \frac{ \| \hat{x}^1\|^2 + \| \hat{x}^3\|^2 }{2}}=1 \;\mbox{and} \; \sqrt{\|\hat{x}^2\|\|\hat{x}^4\|} \le \sqrt{\frac{\|\hat{x}^2\|^2 + \|\hat{x}^4\|^2}{2}}=1.
$$
Thus
\begin{eqnarray*}
\mathcal{T}\left({{\frac{\hat{x}^1}{\|\hat{x}^1\|}} \atop{\frac{\hat{x}^3}{\|\hat{x}^3\|}}},{{\frac{\hat{x}^2}{\|\hat{x}^2\|}} \atop{\frac{\hat{x}^4}{\|\hat{x}^4\|}}},{{\frac{\hat{x}^1}{\|\hat{x}^1\|}} \atop{\frac{\hat{x}^3}{\|\hat{x}^3\|}}},{{\frac{\hat{x}^2}{\|\hat{x}^2\|}} \atop{\frac{\hat{x}^4}{\|\hat{x}^4\|}}}\right)&=& \mathcal{F}\left(\frac{\hat{x}^1}{\|\hat{x}^1\|},\frac{\hat{x}^2}{\|\hat{x}^2\|},\frac{\hat{x}^3}{\|\hat{x}^3\|},\frac{\hat{x}^4}{\|\hat{x}^4\|}\right)\\
&=&\frac{\mathcal{F}(\hat{x}^1,\hat{x}^2,\hat{x}^3,\hat{x}^4)}{\|\hat{x}^1\|\|\hat{x}^2\|\|\hat{x}^3\|\|\hat{x}^4\|}\\
&\ge&\mathcal{F}(\hat{x}^1,\hat{x}^2,\hat{x}^3,\hat{x}^4)\\
&=&\mathcal{T}\left({{\hat{x}^1} \atop{\hat{x}^3}},{{\hat{x}^2} \atop{\hat{x}^4}},{{\hat{x}^1} \atop{\hat{x}^3}},{{\hat{x}^2} \atop{\hat{x}^4}}\right).
\end{eqnarray*}
To summarize, we have found a feasible solution $\left(\frac{\hat{x}^1}{\|\hat{x}^1\|} , \frac{\hat{x}^2}{\|\hat{x}^2\|}, \frac{\hat{x}^3}{\|\hat{x}^3\|}, \frac{\hat{x}^4}{\|\hat{x}^4\|}\right)$ of~\eqref{prob:biquadratic-sphere-a}, which is optimal to its relaxation~\eqref{prob:biquadratic-sphere-b} and thus this relaxation is tight.
\end{proof}

By letting $y= \left({{x^1} \atop{x^3}}\right)$, $z=\left({{x^2} \atop{x^4}}\right)$ and using some scaling technique, we can see that problem~\eqref{prob:biquadratic-sphere-b} share the same solution with
\begin{equation*}
\begin{array}{ll} \max & \mathcal{T}\left(y,z,y,z\right) \\ \st & \|y\|=1,\\&\|z\| =1,
\end{array}
\end{equation*}
which was studied in Subsection~\ref{bi-quadratic-PCA}.

\subsection{Even order multi-linear PCA}
The above discussion can be extended to the even order multi-linear PCA problem:
\begin{equation}\label{prob:multi-linear-pca}
\begin{array}{ll} \max & \mathcal{F}(x^1,x^2,\cdots,x^{2d}) \\ \st & x^i \in \RR^{n_i}, \|x^i\|=1,\;\forall\; i=1,2,\ldots,2d,
\end{array}
\end{equation}
where $\mathcal{F} \in \Br^{n^1 \times \cdots \times n^{2d}}$. An immediate relaxation of~\eqref{prob:multi-linear-pca} is the following
\begin{equation}\label{prob:multi-linear-pca-b}
\begin{array}{ll} \max & \mathcal{F}(x^1,x^2,\cdots,x^{2d}) \\ \st & x^i \in \RR^{n_i}, \sqrt{\sum\limits_{i}^{2d}\|x^i\|^2 }=\sqrt{2d}.
\end{array}
\end{equation}
The following result shows that these two problems are actually equivalent.
\begin{proposition}It holds that problem~\eqref{prob:multi-linear-pca} is equivalent to~\eqref{prob:multi-linear-pca-b}.
\end{proposition}
\begin{proof} It suffices to show that relaxation~\eqref{prob:multi-linear-pca-b} is tight. To this end, suppose $(\hat{x}^1,\cdots, \hat{x}^{2d})$ is an optimal solution of~\eqref{prob:multi-linear-pca-b}. Then $\mathcal{F}(\hat{x}^1,\hat{x}^2,\cdots,\hat{x}^{2d}) >0$ and so $\hat{x}^i \neq 0$ for $i=1,\ldots, 2d$.
We also notice
$$
\sqrt{\bigg{(}\prod\limits_{i=1}^{2d}\|\hat{x}_i\|^2\bigg{)}^{\frac{1}{2d}}  } \le \sqrt{{\sum\limits_{i}^{2d}\|\hat{x}^i\|^2}/{2d}}=1.
$$
Consequently, $\prod\limits_{i=1}^{2d}\|\hat{x}_i\| \le 1$ and
$$
\mathcal{F}\left(\frac{\hat{x}^1}{\|\hat{x}^1 \|},\frac{\hat{x}^2}{\|\hat{x}^2 \|},\cdots,\frac{\hat{x}^{2d}}{\|\hat{x}^{2d} \|}\right)=\frac{\mathcal{F}(\hat{x}^1,\hat{x}^2,\cdots,\hat{x}^{2d})}{\prod\limits_{i=1}^{2d}\|\hat{x}_i\|} \ge \mathcal{F}(\hat{x}^1,\hat{x}^2,\cdots,\hat{x}^{2d}).
$$
Therefore, we have found a feasible solution $\left(\frac{\hat{x}^1}{\|\hat{x}^1 \|},\frac{\hat{x}^2}{\|\hat{x}^2 \|},\cdots,\frac{\hat{x}^{2d}}{\|\hat{x}^{2d} \|}\right)$ of~\eqref{prob:multi-linear-pca}, which is optimal to~\eqref{prob:multi-linear-pca-b} implying that the relaxation is tight.
\end{proof}
We now focus on~\eqref{prob:multi-linear-pca-b}. Based on $\mathcal{F}$, we can construct a larger tensor $\mathcal{G}$ as follows
$$
\mathcal{G}_{i_1\cdots i_{2d}}=\left\{
\begin{array}{cl} \mathcal{F}_{j_1\cdots j_{2d}}, &\mbox{if}\;1+\sum\limits_{\ell =1}^{k-1}n_{\ell}\le i_k \le \sum\limits_{\ell =1}^{k}n_{\ell}
\mbox{ and } j_k=i_k-\sum\limits_{\ell =1}^{k-1}n_{\ell}
\\
0,& \mbox{otherwise}.
 \end{array}\right.
$$
By this construction, we have
$$ \mathcal{F}(x^1,x^2,\cdots,x^{2d})  = \mathcal{G}(\underbrace{y,\cdots,y}_{2d})$$
with $y=((x^1)^{\T},(x^2)^{\T},\cdots,(x^{2d})^{\T})^{\T}$.
We can further symmetrize $\mathcal{G}$ and find a super-symmetric $\mathcal{T}$ such that
\[
\mathcal{T}_{i_1\cdots i_{2d}} := \frac{1}{|\pi(i_1\cdots i_{2d})|}\sum_{j_1\cdots j_{2d} \in\pi(i_1\cdots i_{2d})} \mathcal{G}_{j_1\cdots j_{2d}},\quad\forall\,1\le i_1,\cdots,i_{2d} \le \sum\limits_{\ell =1}^{2d}n_{\ell},
\]
and
$$
\mathcal{T}(\underbrace{y,\cdots,y}_{2d})=\mathcal{G}(\underbrace{y,\cdots,y}_{2d})=\mathcal{F}(x^1,x^2,\cdots,x^{2d}).
$$
Therefore, problem~\eqref{prob:multi-linear-pca-b} is equivalent to
\begin{equation*}
\begin{array}{ll} \max & \mathcal{T}(\underbrace{y,\cdots,y}_{2d}) \\ \st & \left\|y\right\|=\sqrt{2d},
\end{array}
\end{equation*}
which is further equivalent to
\begin{equation*}
\begin{array}{ll} \max & \mathcal{T}(\underbrace{y,\cdots,y}_{2d}) \\ \st & \left\|y\right\|=1.
\end{array}
\end{equation*}
Thus the methods we developed for solving~\eqref{prob:tensor-pca} can be applied to solve~\eqref{prob:multi-linear-pca}.

\subsection{Odd degree tensor PCA}
The last problem studied in this section is the following odd degree tensor PCA problem:
\begin{equation}\label{prob:odd-tensor-pca}
\begin{array}{ll} \max & \mathcal{F}(\underbrace{x,\cdots,x}_{2d+1}) \\ \st & \|x\|=1,
\end{array}
\end{equation}
where $\mathcal{F}$ is a $(2d+1)$-th order super-symmetric tensor. As the degree is odd,
$$
\max_{\|x\|=1} \mathcal{F}(\underbrace{x,\cdots,x}_{2d+1})=\max_{\|x\|=1} |\mathcal{F}(\underbrace{x,\cdots,x}_{2d+1})|=\max_{\|x^i\|_2^2=1,\;i=1,\ldots,2d+1}| \mathcal{F}(x^1,\cdots,x^{2d+1})|,
$$
where the last identity is due to Corollary 4.2 in~\cite{CHLZ11}. The above formula combined with the fact that
$$
\max_{\|x\|=1} |\mathcal{F}(\underbrace{x,\cdots,x}_{2d+1})|\le \max_{\|x\|=1,\;\|y\|=1} |\mathcal{F}(\underbrace{x,\cdots,x}_{2d},y)| \le \max_{\|x^i\|=1,\;i=1,\ldots,2d+1}| \mathcal{F}(x^1,\cdots,x^{2d+1})|
$$
implies
$$
\max_{\|x\|=1} \mathcal{F}(\underbrace{x,\cdots,x}_{2d+1})=\max_{\|x\|=1,\;\|y\|=1} |\mathcal{F}(\underbrace{x,\cdots,x}_{2d},y)| =\max_{\|x\|=1,\;\|y\|=1} \mathcal{F}(\underbrace{x,\cdots,x}_{2d},y).
$$
By using the similar technique as in Subsection~\ref{tri-linear}, problem~\eqref{prob:odd-tensor-pca} is equivalent to an even order tensor PCA problem:
\begin{equation*}
\begin{array}{ll} \max & \mathcal{G}(\underbrace{x,\cdots,x}_{4d}) \\ \st & \|x\|=1,
\end{array}
\end{equation*}
where $\mathcal{G}$ is super-symmetric with
$$\mathcal{G}_{i_1,\cdots,i_{4d}}= \frac{1}{|\pi(i_1\cdots i_{4d})|} \sum\limits_{k=1}^{n}\left(\sum\limits_{j_1\cdots j_{4d}\in \pi(i_1\cdots i_{4d})}\mathcal{F}_{i_1\cdots i_{2d}k}\,\mathcal{F}_{i_{2d+1}\cdots i_{4d}k}\right).$$

\section{Conclusions}

Tensor principal component analysis is an emerging area of research with many important applications in image processing, data analysis, statistical learning, and bio-informatics. In this paper we introduced a new matricization scheme, which ensures that if the tensor is of rank one (in the sense of CP rank), then its matricization is a rank-one matrix, and vice versa. This enables one to apply the methodology in compressive sensing and matrix rank minimization, in particular the $L_1$-norm and nuclear norm optimization techniques. As it turns out, this approach {is very likely to} yield a rank-one solution. This effectively {finds the leading PC} by convex optimization, at least for randomly generated problem instances. The resulting convex optimization model is still large in general. We proposed to use the first-order method such as the ADMM method, which turns out to be very effective in this case. {Multiple principal components can be computed sequentially via the so-called ``deflation'' technique.}

{
{\bf Acknowledgements.} We would like to thank Fei Wang and Yiju Wang for sharing with us their codes of Z-eigenvalue methods.}
\bibliographystyle{plain}
\bibliography{All}

\begin{thebibliography}{10}

\bibitem{Alizadeh93interiorpoint}
F.~Alizadeh.
\newblock Interior point methods in semidefinite programming with applications
  to combinatorial optimization.
\newblock {\em SIAM Journal on Optimization}, 5:13--51, 1993.

\bibitem{BV08}
L.~Bloy and R.~Verma.
\newblock On computing the underlying fiber directions from the diffusion
  orientation distribution function.
\newblock In {\em Medical Image Computing and Computer-Assisted Intervention,
  MICCAI 2008, D.Metaxas, L. Axel, G. Fichtinger and G. Sz¡äekeley, eds.},
  2008.

\bibitem{Boyd-etal-ADM-survey-2011}
S.~Boyd, N.~Parikh, E.~Chu, B.~Peleato, and J.~Eckstein.
\newblock Distributed optimization and statistical learning via the alternating
  direction method of multipliers.
\newblock {\em Foundations and Trends in Machine Learning}, 2011.

\bibitem{Candes-Recht-2008}
E.~J. Cand\`es and B.~Recht.
\newblock Exact matrix completion via convex optimization.
\newblock {\em Foundations of Computational Mathematics}, 9:717--772, 2009.

\bibitem{Candes-Romberg-Tao-2006}
E.~J. Cand\`es, J.~Romberg, and T.~Tao.
\newblock Robust uncertainty principles: {E}xact signal reconstruction from
  highly incomplete frequency information.
\newblock {\em IEEE Transactions on Information Theory}, 52:489--509, 2006.

\bibitem{Candes-Tao-2009}
E.~J. Cand\`es and T.~Tao.
\newblock The power of convex relaxation: near-optimal matrix completion.
\newblock {\em IEEE Transactions on Information Theory}, 56(5):2053--2080,
  2009.

\bibitem{carroll-70}
J.~D. Carroll and J.~J. Chang.
\newblock Analysis of individual differences in multidimensional scaling via an
  n-way generalization of "eckart-young" decomposition.
\newblock {\em Psychometrika}, 35(3):283--319, 1970.

\bibitem{CRPW12}
V.~Chandrasekaran, P.~A Recht, B.and~Parrilo, and A.~S. Willsky.
\newblock The convex geometry of linear inverse problems.
\newblock {\em Foundations of Computational Mathematics}, 12(6):805--849, 2012.

\bibitem{C12}
B.~Chen.
\newblock {\em Optimization with Block Variables: Theory and Applications}.
\newblock PhD thesis, The Chinese Univesrity of Hong Kong, 2012.

\bibitem{CHLZ11}
B.~Chen, S.~He, Z.~Li, and S.~Zhang.
\newblock Maximum block improvement and polynomial optimization.
\newblock {\em SIAM Journal on Optimization}, 22:87--107, 2012.

\bibitem{CGLM08}
P.~Comon, G.~Golub, L.~H. Lim, and B.~Mourrain.
\newblock Symmetric tensors and symmetric tensor rank.
\newblock {\em SIAM Journal on Matrix Analysis and Applications},
  30(3):1254--1279, 2008.

\bibitem{Donoho-06}
D.~Donoho.
\newblock Compressed sensing.
\newblock {\em IEEE Transactions on Information Theory}, 52:1289--1306, 2006.

\bibitem{Douglas-Rachford-56}
J.~Douglas and H.~H. Rachford.
\newblock On the numerical solution of the heat conduction problem in 2 and 3
  space variables.
\newblock {\em Transactions of the American Mathematical Society}, 82:421--439,
  1956.

\bibitem{Eckstein-thesis-89}
J.~Eckstein.
\newblock {\em Splitting methods for monotone operators with applications to
  parallel optimization}.
\newblock PhD thesis, Massachusetts Institute of Technology, 1989.

\bibitem{Eckstein-Bertsekas-1992}
J.~Eckstein and D.~P. Bertsekas.
\newblock On the {D}ouglas-{R}achford splitting method and the proximal point
  algorithm for maximal monotone operators.
\newblock {\em Mathematical Programming}, 55:293--318, 1992.

\bibitem{Fortin-Glowinski-1983}
M.~Fortin and R.~Glowinski.
\newblock {\em Augmented Lagrangian methods: applications to the numerical
  solution of boundary-value problems}.
\newblock North-Holland Pub. Co., 1983.

\bibitem{Gabay-83}
D.~Gabay.
\newblock Applications of the method of multipliers to variational
  inequalities.
\newblock In M.~Fortin and R.~Glowinski, editors, {\em Augmented Lagrangian
  Methods: Applications to the Solution of Boundary Value Problems}.
  North-Hollan, Amsterdam, 1983.

\bibitem{Gandy-Recht-Yamada-2011}
S.~Gandy, B.~Recht, and I.~Yamada.
\newblock Tensor completion and low-n-rank tensor recovery via convex
  optimization.
\newblock {\em Inverse Problems}, 2011.

\bibitem{GTDPMR08}
A.~Ghosh, E.~Tsigaridas, M.~Descoteaux, P.~Comon, B.~Mourrain, and R.~Deriche.
\newblock A polynomial based approach to extract the maxima of an antipodally
  symmetric spherical function and its application to extract fiber directions
  from the orientation distribution function in diffusion mri.
\newblock In {\em Computational Diffusion MRI Workshop (CDMRI'08), New York},
  2008.

\bibitem{Glowinski-LeTallec-89}
R.~Glowinski and P.~Le~Tallec.
\newblock {\em Augmented Lagrangian and Operator-Splitting Methods in Nonlinear
  Mechanics}.
\newblock SIAM, Philadelphia, Pennsylvania, 1989.

\bibitem{GoemansWilliamson1995}
M.~X. Goemans and D.~P. Williamson.
\newblock Improved approximation algorithms for maximum cut and satisfiability
  problems using semidefinite programming.
\newblock {\em Journal of the ACM (JACM)}, 42(6):1115--1145, 1995.

\bibitem{Goldstein-Osher-08}
T.~Goldstein and S.~Osher.
\newblock The split {Bregman} method for {L1}-regularized problems.
\newblock {\em SIAM Journal on Imaging Sciences}, 2:323--343, 2009.

\bibitem{cvx}
M.~Grant and S.~Boyd.
\newblock {CVX: Matlab software for disciplined convex programming, version
  1.21}.
\newblock {\em http://cvxr.com/cvx}, May 2010.

\bibitem{harshman-70}
R.~A. Harshman.
\newblock Foundations of the parafac procedure: models and conditions for an"
  explanatory" multimodal factor analysis.
\newblock 1970.

\bibitem{H90}
J.~H{\aa}stad.
\newblock Tensor rank is {NP}-complete.
\newblock {\em Journal of Algorithms}, 11:644--654, 1990.

\bibitem{HLL09}
D.~Henrion, J.~B. Lasserre, and J.~Loefberg.
\newblock {GloptiPoly 3: Moments, optimization and semidefinite programming}.
\newblock {\em Optimization Methods and Software}, 24:761--779, 2009.

\bibitem{HS10}
J.~J. Hilling and A.~Sudbery.
\newblock The geometric measure of multipartite entanglement and the singular
  values of a hypermatrix.
\newblock {\em Journal of Mathematical Physics}, 51:072102, 2010.

\bibitem{hitchcock-27a}
F.~L. Hitchcock.
\newblock {\em The expression of a tensor or a polyadic as a sum of products}.
\newblock Institute of Technology, 1927.

\bibitem{hitchcock-27b}
F.~L. Hitchcock.
\newblock Multiple invariants and generalized rank of a p-way matrix or tensor.
\newblock {\em Journal of Mathematical Physics}, 7(1):39--79, 1927.

\bibitem{HQ12}
S.~Hu and L.~Qi.
\newblock Algebraic connectivity of an even uniform hypergraph.
\newblock {\em Journal of Combinatorial Optimization}, 24(4):564--579, 2012.

\bibitem{KR02}
E.~Kofidis and P.~A. Regalia.
\newblock On the best rank-$1$ approximation of higher-order supersymmetric
  tensors.
\newblock {\em SIAM Journal on Matrix Analysis and Applications}, 23:863--884,
  2002.

\bibitem{KB09}
T.~G. Kolda and B.~W. Bader.
\newblock Tensor decompositions and applications.
\newblock {\em SIAM Review}, 51:455--500, 2009.

\bibitem{KM11}
T.~G. Kolda and J.~R. Mayo.
\newblock Shifted power method for computing tensor eigenpairs.
\newblock {\em SIAM J. Matrix Analysis}, 32:1095--1124, 2011.

\bibitem{K89}
J.~B. Kruskal.
\newblock Rank, decomposition, and uniqueness for 3-way and n-way arrays.
\newblock {\em Multiway data analysis}, pages 7--18, 1989.

\bibitem{L01}
J.~B. Lasserre.
\newblock Global optimization with polynomials and the problem of moments.
\newblock {\em SIAM Journal on Optimization}, 11:796--817, 2001.

\bibitem{L01_2}
J.~B. Lasserre.
\newblock Polynomials nonnegative on a grid and discrete representations.
\newblock {\em Transactions of the American Mathematical Society},
  354:631--649, 2001.

\bibitem{LN11}
W.~Li and M.~Ng.
\newblock Existence and uniqueness of stationary probability vector of a
  transition probability tensor.
\newblock Technical report, Department of Mathematics, The Hong Kong Baptist
  University, March 2011.

\bibitem{Lim05}
L.~H. Lim.
\newblock Singular values and eigenvalues of tensors: a variational approach.
\newblock In {\em Computational Advances in Multi-Sensor Adaptive Processing,
  2005 1st IEEE International Workshop on}, pages 129--132. IEEE, 2005.

\bibitem{LNWY09}
C.~Ling, J.~Nie, L.~Qi, and Y.~Ye.
\newblock Biquadratic optimization over unit spheres and semidefinite
  programming relaxations.
\newblock {\em SIAM Journal on Optimization}, 20:1286--1310, 2009.

\bibitem{Lions-Mercier-79}
P.~L. Lions and B.~Mercier.
\newblock Splitting algorithms for the sum of two nonlinear operators.
\newblock {\em SIAM Journal on Numerical Analysis}, 16:964--979, 1979.

\bibitem{LMWY09}
J.~Liu, P.~Musialski, P.~Wonka, and J.~Ye.
\newblock Tensor completion for estimating missing values in visual data.
\newblock In {\em The Twelfth IEEE International Conference on Computer
  Vision}, 2009.

\bibitem{Ma-SPCA-2011-submit}
S.~Ma.
\newblock Alternating direction method of multipliers for sparse principal
  component analysis.
\newblock {\em preprint}, 2011.

\bibitem{Ma-Goldfarb-Chen-2008}
S.~Ma, D.~Goldfarb, and L.~Chen.
\newblock Fixed point and {B}regman iterative methods for matrix rank
  minimization.
\newblock {\em Mathematical Programming Series A}, 128:321--353, 2011.

\bibitem{Mackey-NIPS-2008}
L.~Mackey.
\newblock Deflation methods for sparse {PCA}.
\newblock In {\em Advances in Neural Information Processing Systems (NIPS)},
  2008.

\bibitem{P00}
P.~A. Parrilo.
\newblock {\em Structured Semidefinite Programs and Semialgebraic Geometry
  Methods in Robustness and Optimization}.
\newblock PhD thesis, California Institute of Technology, 2000.

\bibitem{P03}
P.~A. Parrilo.
\newblock Semidefinite programming relaxations for semialgebraic problems.
\newblock {\em Mathematical Programming, Series B}, 96:293--320, 2003.

\bibitem{Peaceman-Rachford-55}
D.~H. Peaceman and H.~H. Rachford.
\newblock The numerical solution of parabolic elliptic differential equations.
\newblock {\em SIAM Journal on Applied Mathematics}, 3:28--41, 1955.

\bibitem{Q05}
L.~Qi.
\newblock Eigenvalues of a real supersymmetric tensor.
\newblock {\em Journal of Symbolic Computation}, 40:1302--1324, 2005.

\bibitem{QWW09}
L.~Qi, F.~Wang, and Y.~Wang.
\newblock Z-eigenvalue methods for a global polynomial optimization problem.
\newblock {\em Mathematical Programming, Series A}, 118:301--316, 2009.

\bibitem{QYW10}
L.~Qi, G.~Yu, and E.~X. Wu.
\newblock Higher order positive semi-definite diffusion tensor imaging.
\newblock {\em SIAM Journal on Imaging Sciences}, pages 416--433, 2010.

\bibitem{Recht-Fazel-Parrilo-2007}
B.~Recht, M.~Fazel, and P.~Parrilo.
\newblock Guaranteed minimum-rank solutions of linear matrix equations via
  nuclear norm minimization.
\newblock {\em SIAM Review}, 52(3):471--501, 2010.

\bibitem{Scheinberg-Ma-Goldfarb-NIPS-2010}
K.~Scheinberg, S.~Ma, and D.~Goldfarb.
\newblock Sparse inverse covariance selection via alternating linearization
  methods.
\newblock In {\em NIPS}, 2010.

\bibitem{Tao-Yuan-SPCP-2011}
M.~Tao and X.~Yuan.
\newblock Recovering low-rank and sparse components of matrices from incomplete
  and noisy observations.
\newblock {\em SIAM Journal on Optimization}, 21:57--81, 2011.

\bibitem{TSHK2011}
R.~Tomioka, T.~Suzuki, K.~Hayashi, and H.~Kashima.
\newblock Statistical performance of convex tensor decomposition.
\newblock {\em Advances in Neural Information Processing Systems (NIPS)}, page
  137, 2011.

\bibitem{VandenbergheBoyd1996}
L.~Vandenberghe and S.~Boyd.
\newblock Semidefinite programming.
\newblock {\em SIAM Review}, 38(1):49--95, 1996.

\bibitem{WA04}
H.~Wang and N.~Ahuja.
\newblock Compact representation of multidimensional data using tensor rank-one
  decomposition.
\newblock In {\em Proceedings of the Pattern Recognition, 17th International
  Conference on ICPR}, 2004.

\bibitem{Wang-Yang-Yin-Zhang-2008}
Y.~Wang, J.~Yang, W.~Yin, and Y.~Zhang.
\newblock A new alternating minimization algorithm for total variation image
  reconstruction.
\newblock {\em SIAM Journal on Imaging Sciences}, 1(3):248--272, 2008.

\bibitem{Wen-Goldfarb-Yin-2009}
Z.~Wen, D.~Goldfarb, and W.~Yin.
\newblock Alternating direction augmented {L}agrangian methods for semidefinite
  programming.
\newblock {\em Mathematical Programming Computation}, 2:203--230, 2010.

\bibitem{Yang-Zhang-2009}
J.~Yang and Y.~Zhang.
\newblock Alternating direction algorithms for $\ell_1$ problems in compressive
  sensing.
\newblock {\em SIAM Journal on Scientific Computing}, 33(1):250--278, 2011.

\bibitem{Yuan-2009}
X.~Yuan.
\newblock Alternating direction methods for sparse covariance selection.
\newblock {\em Journal of Scientific Computing}, 51:261--273, 2012.

\end{thebibliography}

\end{document}